\documentclass[12pt,reqno,a4paper]{amsart}

\usepackage{mathtools,color, mathabx}

\usepackage[colorlinks=true, linkcolor=blue, urlcolor=blue, citecolor=blue, pagebackref=true]{hyperref}

\usepackage[abbrev,bibtex-style]{amsrefs}

\usepackage{amssymb}

\usepackage{amsfonts}

\numberwithin{equation}{section} \theoremstyle{plain}
\newtheorem{theorem}{Theorem}
\newtheorem{proposition}[theorem]{Proposition}
\newtheorem{lemma}[theorem]{Lemma}
\newtheorem{corollary}[theorem]{Corollary}

\theoremstyle{definition}
\newtheorem{remark}[theorem]{Remark}

\renewcommand{\leq}{\leqslant}
\renewcommand{\geq}{\geqslant}

\newcommand\Z{\mathbb{Z}}
\newcommand\R{\mathbb{R}}

\newcommand\C{\mathbb{C}}

\newcommand\Q{\mathbb{Q}}

\newcommand\supp{\operatorname{supp}}

\renewcommand\P{\mathbb{P}}

\newcommand\esssup{\operatorname{ess\: sup}}

\renewcommand{\a}{\alpha}
\renewcommand{\b}{\beta}
\renewcommand{\d}{\delta}

\newcommand{\e}{\varepsilon}

\renewcommand{\l}{\lambda}

\newcommand{\s}{\sigma}

\newcommand{\wt}{\widetilde}

\newcommand{\cP}{\mathcal{P}}

\newcommand{\cX}{\mathcal{X}}
\newcommand{\cY}{\mathcal{Y}}

\newcommand\eps{\varepsilon}

\begin{document}

\title[On the dimension of Bernoulli convolutions]{On the dimension of Bernoulli convolutions}

\author{Emmanuel Breuillard}
\address{Centre for Mathematical Sciences\\
Wilberforce Road\\
Cambridge CB3 0WA\\
UK }
\email{efjb2@cam.ac.uk}

\author{P\'eter P. Varj\'u}
\address{Centre for Mathematical Sciences\\
Wilberforce Road\\
Cambridge CB3 0WA\\
UK }
\email{pv270@dpmms.cam.ac.uk}

\thanks{EB acknowledges support from ERC Grant no. 617129 `GeTeMo';
PV acknowledges support from the Royal Society.}



\begin{abstract}
The Bernoulli convolution with parameter $\lambda\in(0,1)$
is the probability measure $\mu_\l$ that is the law of the random variable $\sum_{n\ge0}\pm\lambda^n$,
where the signs are independent unbiased coin tosses.

We prove that each parameter $\lambda\in(1/2,1)$ with $\dim\mu_\lambda<1$
can be approximated by algebraic parameters $\eta\in(1/2,1)$ within
an error of order $\exp(-\deg(\eta)^{A})$ such that $\dim\mu_\eta<1$, for any number $A$.
As a corollary, we conclude that $\dim\mu_\l=1$ for each of $\l=\ln 2, e^{-1/2}, \pi/4$.
These are the first explicit examples of such transcendental parameters.
Moreover, we show that Lehmer's conjecture implies the existence of a constant $a<1$
such that $\dim\mu_\l=1$ for all $\l\in(a,1)$.
\end{abstract}

\maketitle

\setcounter{tocdepth}{1}
\tableofcontents

\section{Introduction}

Let $\l\in(0,1)$ be a real number and let $\xi_0,\xi_1,\ldots$ be a sequence of independent
random variables with $\P(\xi_n=1)=\P(\xi_n=-1)=1/2$.
We define the Bernoulli convolution $\mu_{\l}$ with parameter $\l$ as the law of the random variable
$\sum_{n=0}^\infty \xi_n\l^n$.

This paper is concerned with the problem of determining the set of parameters
$\l$ such that $\dim\mu_{\l}<1$.
It turns out (see \cite{feng-hu}*{Theorem 2.8})
that $\mu_\l$ is always exact dimensional, that is, there is a number
$0\le \a\le 1$ such that
\begin{equation}\label{eq:dim-def}
\lim_{r\to 0}\frac{\log \mu_\l(x-r,x+r)}{\log r}=\a
\end{equation}
for $\mu_\l$-almost every $x$.
We call $\a$ the (local) dimension of $\mu_\l$ and denote this number by $\dim\mu_\l$.

The main result of this paper is the following.
We denote by $\cP_d$ the set of polynomials of degree at most $d$
all of whose coefficients are $-1$, $0$ or $1$.
We write
\[
E_{d,\a}=\{\eta\in(1/2,1): \dim\mu_\eta<\a \text{ and  $P(\eta)=0$ for some $P\in \cP_d$}\}.
\]

\begin{theorem}\label{th:main}
Let $\l\in(1/2,1)$ be such that $\dim \mu_\l<1$.

Then for every $\e>0$, there is a number $A>0$ such that the following holds.
For every sufficiently large integer $d_0$, there is an integer
\[
d\in [d_0,\exp^{(5)}(\log^{(5)}(d_0)+A)]
\]
and there is $\eta\in E_{d, \dim\mu_\l+\e}$ such that
\[
|\l- \eta|\le \exp(-d^{\log^{(3)} d}).
\]
\end{theorem}

In this paper, the base of the $\log$ and $\exp$ functions are $2$;
however, in most places this normalization makes no difference.
When we want to use the natural base, we use the notation $\ln$ and $e^{(\cdot)}$.
We denote by $\log^{(a)}$ and $\exp^{(a)}$ the $a$-fold iteration of the
$\log$ and $\exp$ functions.

Theorem \ref{th:main} has a converse.
\begin{theorem}
Let $\l\in(1/2,1)$ and let $\a<1$.
Suppose that there is a sequence $\{\eta_n\}$
such that $\lim\eta_n=\l$ and $\liminf\dim\mu_{\eta_n}\le\a$ for all $n$.
Then $\dim\mu_\l\le\a$.
\end{theorem}
This is an immediate consequence of the fact that the function
$\l\mapsto \dim\mu_\l$ is lower semi-continuous.
This was proved, for instance, by Hochman and Shmerkin in \cite{HS-local-entropy}*{Theorem 1.8},
but this fact was already known to experts in the area, see the discussion in \cite{HS-local-entropy}*{Section 6}.
We also give a short proof based on our techniques in Section \ref{sc:lower-semi}.

We formulate some corollaries.

\begin{corollary}\label{cr:closure}
We have
\[
\{\l\in(1/2,1):\dim\mu_\l<1\}\subseteq\overline{\{\l\in\overline{\Q}\cap(1/2,1):\dim\mu_\l<1\}},
\]
where $\overline \Q$ is the set of algebraic numbers
and $\overline{\{\cdot\}}$ denotes the closure of the set with respect to the
natural topology of  real numbers.
\end{corollary}

We note that the only known examples of parameters $\l\in(1/2,1)$
such that $\dim\mu_\l<1$ are the inverses of Pisot numbers
(see \cite{Garsia-entropy}*{Theorem I.2} together with \cite{feng-hu}*{Theorem 2.8}
and \cite{You-dimension-entropy}*{Theorem 4.4}), that is algebraic integers all of whose Galois conjugates are
inside the open unit disk.
The set of Pisot numbers is closed (see \cite{Sal-Pisot-closed}).
If one were able to prove that there are no more algebraic parameters with
the property  $\dim\mu_\l<1$, then this would follow
also for transcendental parameters from our result.

The dimension of Bernoulli convolutions for algebraic parameters has been studied
in the paper \cite{BV-entropy}. Recall that Lehmer's conjecture states that there is some numerical constant $\eps_0>0$ such that the Mahler measure $M_\lambda$ (the definition is recalled below in $(\ref{defMahler})$) of every algebraic number $\lambda$ is either $1$ or at least $1+\eps_0$.
It was proved in \cite{BV-entropy} that Lehmer's conjecture implies that there exists
a number $a<1$ such that $\dim\mu_\l=1$ for all algebraic numbers $\l\in(a,1)$. 
We can now drop the condition of algebraicity in that result thanks to Corollary \ref{cr:closure}
and we obtain the following.

\begin{corollary}\label{cr:Lehmer}
If Lehmer's conjecture holds, then there is an absolute constant $a<1$ such that
$\dim\mu_\l=1$ for all $\l\in(a,1)$.
\end{corollary}

We also have the following result.

\begin{corollary}\label{cr:transcendence}
Let $\l\in(1/2,1)$ be a number such that
\begin{equation}\label{eq:transcendence}
|P(\l)|>\exp(-d^{\log^{(3)}d})
\end{equation}
for all $P\in\cP_d$ for all sufficiently large $d$.

Then $\dim \mu_\l=1$.
\end{corollary}

A simple calculation shows that $|P'(x)|<d(d+1)/2$ for all $x\in(0,1)$ and $P\in \cP_d$.
If there is a number $\eta$ that is a root of a polynomial $P\in\cP_d$ such that
\[
|\l-\eta|\le\frac{2}{d(d+1)} \exp(-d^{\log^{(3)}d}),
\]
then $|P(\l)|\le \exp(-d^{\log^{(3)}d})$.
We will see in the proof of Theorem \ref{th:main} that the factor ${2}/{d(d+1)} $
is insignificant and that this slightly stronger approximation also holds
in the setting of the theorem.

There is a large variety of explicit transcendental numbers, for which
the estimate \eqref{eq:transcendence} has been established.
In Sprind\v zuk's classification of numbers, all $\wt S$-numbers, all $\wt T$-numbers
and those $\wt U$-numbers, for which $H_0\ge 2$ satisfy \eqref{eq:transcendence}.
See \cite{Bug-approximation}*{Chapter 8.1} for the notation.

In particular, we have $\dim\mu_\l=1$ for each of
\[
\l\in\{\ln 2, e^{-1/2},\pi/4\}
\]
see e.g. \cite{Wal-transcendence}*{Figure 1},
as well as for many Mahler numbers see e.g. \cite{Zor-Mahler}.
For further examples we refer the reader to the references in \cite{Bug-approximation}*{pp. 189}
and in \cites{Wal-transcendence,Zor-Mahler}.

If one is interested in the smallest possible value that $\dim\mu_\l$ can take
then it is enough to look at algebraic parameters thanks to the following result.
\begin{corollary}\label{cr:minimum}
We have
\[
\min_{\l\in(1/2,1)}\dim\mu_\l=\inf_{\l\in(1/2,1)\cap\overline\Q}\dim\mu_\l.
\]
\end{corollary}
Indeed, let $\dim\mu_{\l_0}=\min_{\l\in(1/2,1)}\dim\mu_\l$.
By Theorem \ref{th:main}, for each $\e>0$, there is an algebraic
parameter $\eta\in(1/2,1)$ such that $\dim \mu_\eta<\dim\mu_{\l_0}+\e$, and this proves the
claim.

Hare and Sidorov \cite{HS-Garsia-entropy} proved that $\dim\mu_\l\ge0.81$ for all Pisot parameters
$\l\in(1/2,1)$.
The authors of that paper explained to us in private communication that their result
can be extended to arbitrary algebraic parameters in $(1/2,1)$.
Combined with Corollary \ref{cr:minimum}, this gives $0.81$ as an
explicit uniform lower bound for the
dimension of  $\mu_\l$ for all parameters in $(1/2,1)$.
\subsection{Background}\label{sc:history}

For thorough surveys on Bernoulli convolutions we refer to \cite{60y}
and \cite{Sol-survey}.
For a discussion of the more recent developments, see \cite{Var-ECM}.

Bernoulli convolutions originate in a paper of Jessen and Wintner \cite{JW-Riemann}
and they have been studied by Erd\H os in \cites{erdos39,erdos}.
If $\l<1/2$, then  $\supp\mu_\l$ is a Cantor set,
and it is easily seen that $\dim \mu_\l=1/\log\l^{-1}$.
(Recall that $\log$ is base $2$ in this paper.)
If $\l=1/2$, then $\mu_\l$ is the normalized Lebesgue measure restricted to the interval $[-2,2]$.

It has been noticed by Erd\H os \cite{erdos39} that $\mu_\l$ may be singular with
respect to the Lebesgue measure even if $\l>1/2$.
In particular, he showed that $\mu_\l$ is singular whenever $\l^{-1}\neq 2$ is a Pisot number.
Moreover, Garsia \cite{Garsia-entropy}*{Theorem I.2} (together with \cite{feng-hu}*{Theorem 2.8} and \cite{You-dimension-entropy}*{Theorem 4.4})
showed that $\dim\mu_\l<1$ if $\l^{-1}\neq 2$ is a Pisot number.

The typical behaviour is absolute continuity for parameters in $(1/2,1)$.
Indeed, Erd\H os \cite{erdos} showed that $\mu_\l$ is absolutely continuous for almost all $\l\in(a,1)$,
where $a<1$ is an absolute constant.
This has been extended by Solomyak \cite{Sol-Bernoulli} to almost all $\l\in(1/2,1)$.

Very recently Hochman \cite{hochman}*{Theorem 1.9}
made a further breakthrough on this problem.
\begin{theorem}[Hochman]\label{th:Hoc}
Let $\l\in(1/2,1)$ be such that $\dim\mu_\l<1$.

Then for every $A>0$, there is a number $d_0$ such that for all integers $d>d_0$,
there is an algebraic number $\eta$ that is a root of a polynomial in $\cP_d$
such that
\[
|\l-\eta|\le \exp(-Ad).
\]
\end{theorem}

In comparison with Theorem \ref{th:main}, Hochman's result has the advantage
that it provides an algebraic approximation of an exceptional parameter
at each scale.
On the other hand, Theorem \ref{th:main} provides a smaller error
and the information that the approximating parameter is also exceptional
(i.e. $\dim\mu_\eta<1$).

Theorem \ref{th:Hoc} also implies that the set of exceptional parameters
\[
\{\l\in(1/2,1):\dim\mu_\l<1\}
\]
is of packing dimension $0$.
Building on this result, Shmerkin \cite{shmerkin} proved that
\[
\{\l\in(1/2,1):\mu_\l\text{ is singular}\}
\]
is of Hausdorff dimension $0$.
We recall that a set of packing dimension $0$ is also a set of Hausdorff dimension $0$.

See also the very recent paper of Shmerkin \cite{Shm-Bernoulli-Lq}, where he
proves a stronger version of Hochman's result for the $L^q$-dimension of Bernoulli convolutions.
He also concludes that outside an exceptional set of Hausdorff dimension $0$ for the parameter,
Bernoulli convolutions are absolutely continuous with a density in $L^q$ for any $q<\infty$.
Moreover, his methods can establish that the density has fractional derivatives.

Theorem \ref{th:Hoc} also implies a conditional result on $\dim\mu_\l$ for transcendental
parameters.
Hochman proved that $\dim \mu_\l=1$ for all transcendental parameters $\l\in(1/2,1)$
if the answer is affirmative
to the following question posed by him \cite{hochman}*{Question 1.10}.
Is there an absolute constant $C>0$ such that
\begin{equation}\label{eq:exponential-sep}
|\eta_1-\eta_2|\ge \exp(-Cd)
\end{equation}
holds for any two different
numbers $\eta_1\neq\eta_2$
that are roots of (not necessarily the same) polynomials in $\cP_d$?
However, such a bound is not yet available;
the best known result in this direction is due to Mahler \cite{Mah-discriminant}*{Theorem 2},
who proved
\begin{equation}\label{eq:Mahler-bound}
|\eta_1-\eta_2|\ge \exp(-Cd\log d),
\end{equation} where $C$ is an absolute constant.
(See Theorem \ref{th:Mahler} below for more details.)

The work of Hochman \cite{hochman} also gives a formula for the dimension of $\mu_\l$,
if $\l$ is
an algebraic number.
Denote by $h_\l$ the entropy of the random walk on the semigroup generated by the
transformations $x\mapsto \l\cdot x+1$ and $x\mapsto \l\cdot x-1$.
More precisely, let
\[
h_\l=\lim_{n\to\infty}\frac{1}{n}H\Big(\sum_{i=0}^{n-1}\xi_i\l^i\Big)=\inf\frac{1}{n}H\Big(\sum_{i=0}^{n-1}\xi_i\l^i\Big),
\]
where $H(\cdot)$ denotes the Shannon entropy of a discrete random variable.
With this notation Hochman's formula is
\begin{equation}\label{eq:Hochman}
\dim \mu_\l=\min(-h_\l/\log\l,1).
\end{equation}
(See \cite{BV-entropy}*{Section 3.4}, where the formula is derived in this form from Hochman's main result.)

The quantity $h_\l$ has been studied in the paper \cite{BV-entropy}.
It was proved there \cite{BV-entropy}*{Theorem 5} that there is an absolute constant $c_0>0$ such that
for any algebraic number, we have
\[
c_0\cdot \min(\log M_\l,1)\le h_\l\le \min(\log M_\l,1).
\]
The $\log$'s in this formula as well as those that appear in the definition of entropy are
base $2$.
Numerical calculations reported in that paper indicate that one can take $c_0=0.44$.
This result combined with Hochman's formula implies that
$\dim\mu_\l=1$ provided $\l$ is an algebraic number with $1>\l>\min(2,M_\l)^{-c_0}$.
Here, and everywhere in the paper, we denote by $M_\l$ the Mahler measure of an algebraic number
$\l$.
That is, if $P(x)=a_d\prod(x-\l_j)$ is the minimal polynomial of $\l$ in $\Z[x]$, then by definition,
\begin{equation}\label{defMahler}
M_\l=|a_d|\prod_{j:|\l_j|>1}|\l_j|.
\end{equation}

\subsection{The strategy of the proof}\label{sc:strategy}
This section gives an informal  account of the proof
of Theorem \ref{th:main}.
All the arguments presented here will be repeated in a rigorous fashion later
in the paper.
Therefore, we take a rather relaxed approach towards our estimates.
In particular, we will write $\lessapprox$ to indicate an inequality that could be made
valid by inserting suitable constants in appropriate places.

The proof of our results builds on the techniques introduced by Hochman
in \cite{hochman} using entropy estimates.

We work with the following notion of entropy.
Let $X$ be a bounded random variable and let $r>0$ be a real number.
We define
\[
H(X;r):=\int_0^1 H(\lfloor X/r+t\rfloor)dt.
\]
On the right hand side, $H(\cdot)$ denotes the Shannon entropy of a discrete random variable.
In addition, we define the conditional entropies
\[
H(X;r_1|r_2):=H(X;r_1)-H(X;r_2).
\]
We will study the basic properties of these quantities in Section \ref{sc:entropy}. In particular $H(X;r)$  is a non-increasing function of $r$. Furthermore $0 \leq H(X;r) \leq \log r^{-1} + O(1)$, where the implied constant depends only on $\esssup |X|$.
By abuse of notation, we write $H(\mu;r_1|r_2)=H(X;r_1|r_2)$ and similar expressions
if $\mu$ denotes the law of $X$.

These quantities differ from those used by Hochman in that they involve an averaging
over a random translation.
This averaging endows these quantities with some useful properties as we will see in Section
\ref{sc:entropy-at-scale}, which often comes in handy.
The idea of this averaging procedure originates in Wang's paper \cite{Wan-quantitative}*{Section 4.1}.

We fix a number $\l\in(1/2,1)$ until the end of the section.
For a set $I\subset\R_{>0}$, we write $\mu_\l^I$
for the law of the random variable
\[
\sum_{n\in\Z:\l^n\in I}\xi_n\l^n.
\]
We note that in this notation $\mu_\l^{(0,1]}=\mu_\l$ and $\mu_\l^{(\l^n,1]}$ is the law of $\sum_{j=0}^{n-1}\xi_j\l^j$, the
first $n$ term truncation of the series defining Bernoulli convolutions.

We note that
\[
\dim\mu_\l=\lim_{n\to\infty}\frac{H(\mu_\l;\l^n)}{n\log\l^{-1}},
\]
see Lemma \ref{lm:dim-entropy},
and that $H(\mu_\l^{(\l^n,1]};\l^n)\approx H(\mu_\l;\l^n)$ up to additive constants independent of $n$.
Hence
\begin{equation}\label{eq:diment}
H(\mu_\l^{(\l^n,1]};\l^n)\approx n\log\l^{-1}\dim\mu_\l.
\end{equation}

We now assume that $\dim\mu_\l<1$
and we assume by contradiction that the algebraic approximations to $\l$
claimed in Theorem \ref{th:main} do not exist.
In the first part of the proof given in Section \ref{sc:initial},
we search for integers $n$ with the property that
\begin{equation}\label{eq:firstaim}
H(\mu^{(\l^n,1]}_\l;r)\ge n\log\l^{-1}(\dim \mu_\l+\e)
\end{equation}
for a suitable scale $r\approx n^{-Cn}$. Equation \eqref{eq:firstaim} is a small improvement over \eqref{eq:diment} when we replace $\l^n$
with the smaller scale $r$.

If $\e>0$  is small enough so that the right hand side of \eqref{eq:firstaim} is $<n$, and if \eqref{eq:firstaim} fails, then there are pairs of choices of the signs in the sum
\[
\sum_{j=0}^{n-1}\pm \l^j
\]
that give the same value within an error of $r$.
For each such pair, there corresponds a non-zero polynomial $P\in\cP_{n-1}$
such that $|2P(\l)|<r$.
In Section \ref{sc:initial}, we show that these polynomials must have a common root
$\eta$ and $|\l-\eta|<n^{-4n}$.
Since this collection of polynomials is rich enough to cause the failure of \eqref{eq:firstaim},
we obtain
\[
H(\mu_\eta^{(\eta^n,1]})\le n\log\l^{-1}(\dim \mu_\l+\e),
\]
which yields $h_\eta\le \log\l^{-1}(\dim\mu_\l+\e)$.
Plugging this into \eqref{eq:Hochman}, we get $\dim\mu_\eta\le\dim\mu_\l+\e'$,
where $\e'$ is arbitrarily close to $\e$ if $n$ is sufficiently large.
Hence $\eta\in E_{n,\dim\mu_\l+\e'}$ .

Then we choose another integer $n'$ such that $|\l-\eta|$ is just slightly larger than
${n'}^{-4n'}$.
If \eqref{eq:firstaim} fails again for $n'$ and for a suitable $r'$, then we can repeat the above
argument to find another number $\eta'\in E_{n',\dim\mu_\l+\e}$ such that $|\l-\eta'|<{n'}^{-4n'}$.
Then $|\eta-\eta'|<2{n'}^{-4n'}$, and we can conclude $\eta=\eta'$ thanks to \eqref{eq:Mahler-bound}
(the result of Mahler on the separation
between roots of polynomials in $\cP_n$).
However, we carefully chose $n'$ to make sure that $|\l-\eta'|<{n'}^{-4n'}<|\l-\eta|$, hence
we cannot have $\eta=\eta'$, which shows that
\eqref{eq:firstaim} must hold for at least one of $n$ or $n'$.

The way we exploited Mahler's bound \eqref{eq:Mahler-bound} is reminiscent to Hochman's argument
for showing $\dim\mu_\l=1$ for all transcendental $\l\in(1/2,1)$ assuming
the stronger bound \eqref{eq:exponential-sep}
discussed in the previous section.

We will use the (indirect) assumption on the lack of algebraic approximations to $\l$
to control $n'$ in terms of $n$.
Indeed, if  \eqref{eq:firstaim} fails for $n$, we get that it holds for $n'$ with
\begin{equation}\label{eq:growth}
{n'}^{4n'}\lessapprox |\l-\eta|^{-1}<\exp(n^{\log^{(3)}n}).
\end{equation}
This will enable us to produce suitably many integers $n$ in a given range such that
\eqref{eq:firstaim} holds.

In the second part of the proof, which we discuss in Section \ref{sc:convolution}, we
use the identity
\[
\mu_\l^{I_1\dot\cup\ldots\dot\cup I_k}=\mu_\l^{I_1}*\ldots*\mu_\l^{I_k}
\]
and argue that entropy increases under convolution to improve on the
bound \eqref{eq:firstaim}.
We use the following result from \cite{Var-Bernoulli-algebraic}*{Theorem 3}.

\begin{theorem}\label{th:low-entropy-convolution}
For every $0<\a\le 1/2$, there are numbers $C,c>0$ such that the following holds.
Let $\mu,\nu$ be two compactly supported probability measures on $\R$.
Let  $\s_2<\s_1<0$ and $0<\b\le 1/2$ be real numbers.
Suppose that
\[
H(\mu;2^\s|2^{\s+1})<1-\a
\]
for all $\s_2<\s<\s_1$.
Suppose further that
\[
H(\nu; 2^{\s_2}|2^{\s_1})>\b(\s_1-\s_2).
\]

Then
\[
H(\mu*\nu; 2^{\s_2}|2^{\s_1})> H(\mu;2^{\s_2}|2^{\s_1})+c\b(\log\b^{-1})^{-1}(\s_1-\s_2)-C.
\]
\end{theorem}

We note that the supremum of the values $H(\mu;r|2r)$ may take over all probability
measures $\mu$ is $1$ (see \eqref{eq:GenEntUpperbnd} below and the comment following it). We will see (in Lemma \ref{lm:dimlessthan1})
that the assumption $\dim\mu_\l<1$ implies that there is
a number $\a>0$ such that $H(\mu_\l^I;r|2r)<1-\a$ for all $r>0$ and for all $I\subset \R_{>0}$.
This means that the hypothesis of Theorem \ref{th:low-entropy-convolution}
holds for $\mu=\mu^{I}_{\l}$  for all $I\subset \R_{>0}$ with an $\a$ depending only on $\l$.

We give a brief and informal explanation on how this result will be used.
Suppose that \eqref{eq:firstaim} holds for some $n$ and $r$.
Now \eqref{eq:diment} implies
\[
H(\mu_\l^{(\l^n,1]};\l^n)<(\dim\mu_\l+\e/2)n\log \l^{-1},
\]
if $n$ is sufficiently large, so we can show that
\[
H(\mu_\l^{(\l^{n},1]};r|\l^n)\ge \e' n,
\]
for some $\e'$ depending only on $\e$ and $\l$.

For simplicity of exposition, we assume now that the stronger
bound
\[
H(\mu_\l^{(\l^{n},1]};r|r^{9/10})\ge \e' n
\]
holds.
There is no way to justify this hypothesis; in the actual proof we need to consider a suitable
decomposition of the scales between $\l^{n}$ and $r$.

Using scaling properties of entropy, we can write
\[
H(\mu_\l^{(\l^{n(j+1)},\l^{nj}]};r\l^{jn}|r^{9/10}\l^{jn})\ge \e' n.
\]
We consider this inequality for $j=0,1,\ldots, N-1$ for some $N\approx (\log r^{-1})/n$ so that
$r^{1/10}\le\l^{jn}\le 1$ for each $j$ in the range.
Hence
\[
H(\mu_\l^{(\l^{n(j+1)},\l^{nj}]};r^{11/10}|r^{9/10})\ge \e' n,
\]
because $[r^{11/10},r^{9/10}]\supset[r\l^{jn},r^{9/10}\l^{jn}]$.

We can now apply Theorem \ref{th:low-entropy-convolution} $N\approx\log n$ times
with
\[
\b\approx \frac{n}{\log r^{-1}}\approx\frac{1}{\log n},
\]
and we obtain
\begin{equation}\label{eq:secondaim}
H(\mu_\l^{(\l^{nN},1]};r^{11/10}|r^{9/10})\gtrapprox \frac{\log r^{-1}}{\log\log n},
\end{equation}
i.e. the average entropy of a digit is at least $\approx(\log\log n)^{-1}$.

Then we will apply Theorem \ref{th:low-entropy-convolution} again in a second stage.
Let $n_1,n_2,\ldots$ be a sequence of integers such that \eqref{eq:firstaim}
and hence \eqref{eq:secondaim} holds.
We apply Theorem \ref{th:low-entropy-convolution} repeatedly again with
$\b_i\approx 1/\log^{(2)}(n_i)$, and find that the average entropy of a digit
between suitable scales is at least
\[
\approx\sum\frac{1}{\log^{(2)}(n_i)\log^{(3)}(n_i)}.
\]
If $n_i$ does not grow faster than $\exp^{(2)}(i\log^{(2)}i)$, then the above sum
can be arbitrarily large contradicting the fact that the entropy of a digit
cannot exceed 1. This contradiction ends the proof. 

Note that using the argument that we presented in the beginning of this sketch, one can show that the lack of the algebraic
approximations claimed in Theorem \ref{th:main}
implies that we can find a sequence $n_i$ that satisfies our requirement \eqref{eq:firstaim}
and also satisfies the growth condition
\[
n_{i+1}^{n_{i+1}}\lessapprox\exp\big(n_i^{\log^{(3)}n_i}\big),
\]
see \eqref{eq:growth}.
We can use this to prove $n_i\lessapprox \exp^{(2)}(i\log^{(2)}i)$ by induction.

\subsection{Notation}

We denote by the letters $c$, $C$ and their indexed variants various constants
that could in principle be computed explicitly following the proofs step by step.
The value of these constants denoted by the same symbol may change between occurrences.
We keep the convention that we denote by lower case letters the constants that are
best thought of as ``small'' and by capital letters the ones that are ``large''.

We denote by $\log$ and $\exp$ the base $2$ logarithm and exponential functions
and write $\ln$ for the  logarithm in base $e$.
We denote by $\log^{(a)}$ and $\exp^{(a)}$ the $a$-fold iterates of the $\log$
and $\exp$ functions.

The letter $\l$ denotes a number in $(0,1)$.
For a bounded set $I\subset \R_{>0}$, we denote by  $\mu_\l^{I}$ the law of the random variable
\[
\sum_{n\in\Z:\l^n\in I}\xi_n\l^n,
\]
where $\xi_n$ is a sequence of independent unbiased $\pm1$ valued random variables.
In particular, we write $\mu_\l=\mu_\l^{(0,1]}$.

We denote by $\cP_d$ the set of polynomials of degree at most $d$ with coefficients
$\pm 1$ and $0$.

\subsection{The organization of this paper}

We begin by discussing some basic properties of entropy in Section \ref{sc:entropy}, which
we will rely on throughout the paper.
Section \ref{sc:initial} contains the first part of the proof of the main result focusing
on the initial entropy estimate \eqref{eq:firstaim} mentioned above.
The proof of Theorem \ref{th:main} is completed in Section \ref{sc:convolution},
where we exploit Theorem \ref{th:low-entropy-convolution} to improve
on our initial entropy estimate.

\subsection*{Acknowledgment}
We are grateful to Yann Bugeaud, Kevin Hare, Mike Hochman, Nikita Sidorov and Evgeniy Zorin for helpful
discussions.
We are also grateful to Mike Hochman for pointing out the converse of Theorem \ref{th:main}.
We thank the anonymous referee, S\'ebastien Gou\"ezel, Nicolas de Saxc\'e and Ariel Rapaport for a very careful
reading of our manuscript and for numerous comments and suggestions that greatly improved
the presentation of our paper.

\section{Preliminaries on entropy}\label{sc:entropy}

The purpose of this section is to provide some background material on entropy.

\subsection{Shannon and differential entropies}
If $X$ is a discrete random variable, we write $H(X)$ for its Shannon entropy, that is
\[
H(X)=\sum_{x\in\cX} -\P(X=x)\log\P(X=x),
\]
where $\cX$ denotes the set of values $X$ takes.
We recall that the base of $\log$ is $2$ throughout the paper.
If $X$ is an absolutely continuous random variable with density $f:\R\to\R_{\ge0}$, we write
$H(X)$ for its differential entropy, that is
\[
H(X)=\int -f(x)\log f(x) dx.
\]
This dual use for $H(\cdot)$ should cause no confusion, as the type
of the random variable will always be clear from the context.
If $\mu$ is a probability measure, we write $H(\mu)=H(X)$, where $X$ is a random variable with law
$\mu$.

Shannon entropy is always non-negative.
Differential entropy on the other hand can take negative values.
For example, if $a \in\R_{>0}$, and $X$ is a random variable
with finite differential entropy, then it follows from the change of variables formula that
\begin{equation}\label{change}
H(aX) = H(X) + \log a,
\end{equation}
which can take negative values when $a$ varies.
On the other hand, if $X$ takes countably many values, the Shannon entropy of $aX$ is the same as that of $X$.
Note that both entropies are invariant under translation by a constant in $\R$.

We define $F(x):=-x\log (x)$ for $x>0$ and recall that
$F$ is concave.
{}From the concavity of $F$ and Jensen's inequality,
we see that for any discrete random variable $X$ taking at most $N$ different values,
\begin{equation}\label{support}
H(X) \leq \log N.
\end{equation}

Let $X$ and $Y$ be two discrete random variables.
We define the conditional entropy of $X$ relative to $Y$ as
\begin{align*}
H(X|Y)=&\sum_{y\in\cY} \P(Y=y) H(X|Y=y)\\
=&\sum_{y\in\cY} \P(Y=y) \sum_{x\in\cX} -\frac{\P(X=x,Y=y)}{\P(Y=y)}\log\frac{\P(X=x,Y=y)}{\P(Y=y)}.
\end{align*}
We recall some well-known properties.
We always have $0\le H(X|Y)\le H(X)$, and $H(X|Y)=H(X)$ if and only if the two random variables are
independent (see \cite{cover-thomas}*{Theorem 2.6.5}).

We recall the following result from \cite{madiman}*{Theorem I}.

\begin{proposition}[Submodularity inequality]
\label{ruzsa}
Assume that $X,Y,Z$ are three independent $\R$-valued random variables
such that the distributions of $Y$, $X+Y$, $Y+Z$ and $X+Y+Z$ are absolutely continuous with respect
to Lebesgue measure and have finite differential entropy.
Then
\begin{equation}\label{ruzsaineq}
H(X+Y+Z) + H(Y) \leq H(X+Y) + H(Y+Z).
\end{equation}
\end{proposition}

This result goes back in some form at least to a
paper by Kaimanovich and Vershik \cite{kaimanovich-vershik}*{Proposition 1.3}.
The version in that paper assumes that the laws of $X$, $Y$ and $Z$ are identical.
The inequality was rediscovered by Madiman \cite{madiman}*{Theorem I} in the
greater generality stated above.
Then it was recast in the context of entropy analogues of sumset
estimates from additive combinatorics by  Tao \cite{tao}
and Kontoyannis and Madiman \cite{kontoyannis-madiman}.
And indeed Proposition \ref{ruzsa} can be seen as an entropy analogue
of the Pl\"unnecke--Ruzsa inequality in additive combinatorics.
For the proof of this exact formulation see \cite{BV-entropy}*{Theorem 7}.

\subsection{Entropy at a given scale}\label{sc:entropy-at-scale}

We recall the notation
\[
H(X;r)=\int_{0}^1 H(\lfloor X/r+t\rfloor) dt
\]
and
\[
H(X;r_1|r_2)=H(X;r_1)-H(X;r_2).
\]

These quantities originate in the work of Wang \cite{Wan-quantitative}, and they also
play an important role in the paper \cite{LV-entropy-sum-product}, where a
quantitative version of Bourgain's sum-product
theorem is proved.

We continue by recording some useful facts about these notions.
If $N$ is an integer then we have the following interpretation, which follows easily from the definition.
\begin{equation}\label{equation:interpret2}
H(X;N^{-1}r|r)=\int_0^1 H(\lfloor N(r^{-1}X+t)\rfloor|\lfloor r^{-1} X+t\rfloor).
\end{equation}
Indeed, $\lfloor r^{-1} X+t\rfloor$ is a function of $\lfloor N(r^{-1}X+t)\rfloor$, hence
\[
H(\lfloor N(r^{-1}X+t)\rfloor|\lfloor r^{-1} X+t\rfloor)
=H(\lfloor N(r^{-1}X+t)\rfloor)-H(\lfloor r^{-1} X+t\rfloor).
\]
Combining this interpretation with \eqref{support} we see that
\begin{equation}\label{eq:GenEntUpperbnd}
H(\mu;r|2r)\le 1
\end{equation}
for any probability measure $\mu$.
This upper bound is best possible, as demonstrated by the uniform measures
on long intervals.

It is immediate from the definitions that we have the scaling formulae
\[
H(sX;sr)=H(X;r),\qquad
H(sX;sr_1|sr_2)=H(X;r_1|r_2),
\]
for any random variable $X$ and real numbers $s,r,r_1,r_2>0$.
In particular, we have
\begin{equation}\label{eq:scaling}
H(\mu_\l^{\l^k I};\l^kr_1|\l^kr_2)=H(\mu_\l^I;r_1|r_2),
\end{equation}
for any integer $k$, real numbers $r_1,r_2>0$ and $I\subset\R_{>0}$.

The next lemma gives an alternative definition for entropy at a given scale.
\begin{lemma}[\cite{Var-Bernoulli-algebraic}*{Lemma 5}]\label{lm:second-def}
Let $X$ be a bounded random variable in $\R$.
Then
\[
H(X;r)=H(X+I_r)-H(I_r)=H(X+I_r)-\log(r).
\]
where $I_r$ is a uniform random variable in $[0,r]$ independent of $X$.
\end{lemma}

It follows from the definition that being an average of Shannon entropies
$H(X;r)$ is always non-negative.
Similarly, we see from \eqref{equation:interpret2} that $H(X;r_1|r_2)$ is also non-negative
if $r_2/r_1$ is an integer.
We will see below that this holds also for any $r_2\ge r_1$.

The next lemma shows that conditional entropy between scales of integral ratio
cannot decrease by taking convolution of measures.

\begin{lemma}[\cite{Var-Bernoulli-algebraic}*{Lemma 6}]\label{lm:entropy-conv-nondecrease}
Let $X$ and $Y$ be two bounded independent random variables in $\R$.
Let $r_2>r_1>0$ be two numbers such that $r_2/r_1\in\Z$.
Then
\[
H(X+Y;r_1|r_2)\ge H(X;r_1|r_2).
\]
\end{lemma}

We record an instance of this lemma that is of particular importance to us.
We have
\begin{equation}\label{eq:subset}
H(\mu^{I_1}_\l;r_1|r_2)\ge H(\mu^{I_2}_\l;r_1|r_2)
\end{equation}
for any $I_2\subset I_1\subset \R_{>0}$ provided the ratio of the scales $r_2/r_1$
is an integer.
Unfortunately, this may fail if the ratio of the scales in non-integral, but we always have
the following.
If $r_2/r_1\ge 2$, we can find $r_1\le t_1\le t_2\le r_2$
such that $t_2/t_1$ is an integer and
\[
H(\mu^{I_2}_\l;t_1|t_2)\ge H(\mu^{I_2}_\l;r_1|r_2)/2.
\]
We combine this with \eqref{eq:subset} and \eqref{eq:subscale} (see below) and get
\begin{equation}\label{eq:subset-noninteger}
H(\mu^{I_1}_\l;r_1|r_2)\ge H(\mu^{I_1}_\l;t_1|t_2)\ge H(\mu^{I_2}_\l;t_1|t_2)
\ge H(\mu^{I_2}_\l;r_1|r_2)/2.
\end{equation}
(It is possible to prove a variant of this with a small additive error term
instead of the multiplicative constant, see \cite{Var-Bernoulli-algebraic}*{Lemma 9}.
However, for the purposes of this paper \eqref{eq:subset-noninteger} is more convenient.)

We recall a result form \cite{LV-entropy-sum-product} (see also  \cite{Var-Bernoulli-algebraic}*{Lemma 8}),
which establishes that
$H(X;r)$ is a monotone increasing and Lipschitz
function of $-\log r$; in particular $H(X;r_1|r_2)$ is nonnegative for all $r_1\le r_2$.

\begin{lemma}\label{lemma:Lipschitz}
Let $X$ be a bounded random variable in $\R$.
Then for any $r_2\ge r_1>0$ we have
\[
0\le H(X;r_1)-H(X;r_2)\le 2(\log r_2-\log r_1).
\]
\end{lemma}

This lemma implies that $H(\mu;r_1|r_2)\ge 0$, whenever $r_2\ge r_1$.
Moreover, we have
\begin{equation}\label{eq:subscale}
H(\mu;r_1|r_2)\ge H(\mu;s_1|s_2)
\end{equation}
provided $(s_1,s_2)\subset(r_1,r_2)$.

\subsection{Bernoulli convolutions of dimension less than 1}\label{sc:dimlessthan1}

The purpose of this section is to show that the entropy of a single digit
for a Bernoulli convolution that is of dimension less than $1$ is bounded away from $1$.
This implies that Theorem \ref{th:low-entropy-convolution} applies to
$\mu=\mu^I_\l$ for any $I$ between any scales
if $\dim\mu_\l<1$ with an $\a$ depending only on $\l$.

\begin{lemma}\label{lm:dimlessthan1}
Let $\l\in(1/2,1)$ be such that $\dim\mu_\l<1$.
Then there is a number $\a>0$ such that
\[
H(\mu_\l;r|2r)<1-\a
\]
for all $r>0$.
\end{lemma}

Thanks to \eqref{eq:scaling} and \eqref{eq:subset},
the same conclusion holds for $\mu_\l^I$ for any $I\subset\R_{>0}$
in place of $\mu_\l$.

We begin by recalling the relation between the dimension and the entropy of
Bernoulli convolutions, which is folklore.

\begin{lemma}\label{lm:dim-entropy}
Let $\l\in (0,1)$.
Then
\[
\dim\mu_\l=\lim_{r\to 0} \frac{H(\mu_\l;r)}{\log r^{-1}}.
\]
\end{lemma}

\begin{proof}
By \cite{feng-hu}*{Theorem 2.8}, $\mu_\l$ is exact dimensional.
By  \cite{You-dimension-entropy}*{Theorem 4.4}, the R\'enyi entropy dimension
of an exact dimensional measure coincides with its local dimension
(the number $\a$ in \eqref{eq:dim-def}).

Thus,
\[
\dim \mu_\l =\lim_{r\to 0} \frac{H(\lfloor r^{-1}X\rfloor)}{\log r^{-1}},
\]
where $X$ is a random variable with law $\mu_\l$.
Moreover, the same formula holds for any translates of $X$,
and the claim follows by dominated convergence.
\end{proof}

We fix $\l\in(1/2,1)$ such that $\dim\mu_\l<1$.
By Lemma \ref{lm:dim-entropy},
there are numbers $N$ and $\a_0>0$ such that
\begin{equation}\label{eq:dimlessthan1}
H(\mu_\l;2^{-n})<n(1-\a_0)
\end{equation}
for any $n>N$.

We assume to the contrary that there is a number $r$
such that
\begin{equation}\label{eq:dimlessthan1_2}
H(\mu_\l;r|2r)\ge 1-\a,
\end{equation}
where $\a>0$ is a suitably small number depending only on $\a_0$
to be specified later.

To contradict \eqref{eq:dimlessthan1}, we aim to produce more
digits of high entropy.
One source of these digits will be the scaling formula \eqref{eq:scaling},
which together with \eqref{eq:subset} implies
\begin{equation}\label{eq:source1}
H(\mu_\l;\l^kr|2\l^kr)\ge H(\mu_\l^{(0,\l^k]};\l^kr|2\l^kr)=H(\mu_\l;r|2r)\ge 1-\a.
\end{equation}

The other source is the next lemma.
\begin{lemma}\label{lm:moredigits}
Let $\mu$ be a compactly supported probability measure on $\R$ and  let $r>0$ be a number.
Then
\[
1-H(\mu;2r|4r)\le4(1-H(\mu;r|2r)).
\]
\end{lemma}
\begin{proof}
Write $\chi_s$ for the uniform probability measure on the interval $[0,s]$
and let $\eta_s=(\d_0+\d_s)/2$, where $\d_x$ denotes the unit mass
supported at the point $x$.
By Lemma \ref{lm:second-def}, we have
\begin{align}
H(\mu;r|2r)=&H(\mu*\chi_r)-H(\chi_r)-H(\mu*\chi_{2r})+H(\chi_{2r})\nonumber\\
=&1-(H(\mu*\eta_r*\chi_r)-H(\mu*\chi_r)).\label{eq:digitentropy}
\end{align}
Note that $\chi_{2r}=\chi_r*\eta_r$.

By submodularity (Proposition \ref{ruzsa}), we have
\[
H(\mu*\eta_r*\eta_r*\chi_r)-H(\mu*\eta_r*\chi_r)
\le H(\mu*\eta_r*\chi_r)-H(\mu*\chi_r),
\]
hence
\[
H(\mu*\eta_r*\eta_r*\chi_r)-H(\mu*\chi_r)\le 2( H(\mu*\eta_r*\chi_r)-H(\mu*\chi_r)).
\]

We note the identity
\[
\eta_r*\eta_r=\frac{\eta_{2r}+\d_r}{2}.
\]
By concavity of $F(x)=-x\log x$, we have
\[
H(\mu*\eta_r*\eta_r*\chi_r)\ge H(\mu*\eta_{2r}*\chi_r)/2+H(\mu*\d_r*\chi_r)/2.
\]
Thus
\begin{align*}
H(\mu*\eta_{2r}*\chi_r)-H(\mu*\chi_r)\le&2(H(\mu*\eta_r*\eta_r*\chi_r)-H(\mu*\chi_r))\\
\le&4( H(\mu*\eta_r*\chi_r)-H(\mu*\chi_r)).
\end{align*}

We use submodularity again to write
\begin{align*}
H(\mu*\eta_{2r}*\chi_{2r})-H(\mu*\chi_{2r})
=&H(\mu*\eta_{2r}*\eta_r*\chi_{r})-H(\mu*\eta_r*\chi_{r})\\
\le &H(\mu*\eta_{2r}*\chi_r)-H(\mu*\chi_r)\\
\le&4( H(\mu*\eta_r*\chi_r)-H(\mu*\chi_r)).
\end{align*}
We combine this with \eqref{eq:digitentropy} and conclude the lemma.
\end{proof}

\begin{proof}[Proof of Lemma \ref{lm:dimlessthan1}]
We assume to the contrary that \eqref{eq:dimlessthan1}
and \eqref{eq:dimlessthan1_2} hold and we
fix two integers $K$, $J$.

Using Lemma \ref{lm:moredigits} repeatedly, we find that
\[
H(\mu_\l;2^kr|2^{k+1}r)\ge1-4^k\a
\]
holds for all $k\in\Z_{\ge_0}$.
We sum these inequalities for $k=0,\ldots, K-1$,
and arrive at
\[
H(\mu_\l;r|2^Kr)\ge K-4^K\a.
\]

We choose an integer $m$ such that $2^{-K-1}\le \l^m\le 2^{-K}$
and use \eqref{eq:source1} together with the above argument to
conclude
\[
H(\mu_\l;\l^{jm}r|2^K\l^{jm}r)\ge K-4^K\a
\]
for any $j\in \Z_{\ge0}$.
We sum this for $j=0,\ldots, J-1$ and use \eqref{eq:subscale}
to get
\[
H(\mu_\l;\l^{(J-1)m}r)\ge J(K-4^K\a).
\]

Since $\l^m\ge 2^{-K-1}$, we get
\[
H(\mu_\l;2^{-(J(K+1)+\lceil\log r^{-1}\rceil)})\ge J(K-4^K\a)
\ge J(K+1)(1-4^K\a-(K+1)^{-1}).
\]

We set the parameters.
We take $K$ to be large enough so that $(K+1)^{-1}<\a_0/3$.
Then we take $\a$ small enough so that $4^K\a<\a_0/3$.
Finally, we take $J$ sufficiently large so that we get a contradiction
to \eqref{eq:dimlessthan1} for $n=J(K+1)+\lceil\log r^{-1}\rceil$.
This proves the lemma.
\end{proof}

\subsection{Lower semi-continuity}
\label{sc:lower-semi}

The purpose of this section is to establish the following result.
\begin{lemma}
The function $\l\mapsto\dim\mu_\l$ is lower semi-continuous.
\end{lemma}

\begin{proof}
By Lemma \ref{lm:dim-entropy}, we have
\[
\dim\mu_\l=\lim_{n\to \infty}\frac{H(\mu_\l;\l^n|1)}{-n\log\l}.
\]

For each $n$, the function
\[
\l\mapsto\frac{H(\mu_\l;\l^n|1)}{-n\log\l}
\]
is continuous.
This follows from \cite{Var-Bernoulli-algebraic}*{Lemma 7} and Lemma \ref{lemma:Lipschitz}.

We show that the sequence of functions 
\[
f_k(\l)=\frac{H(\mu_\l;\l^{2^k}|1)-2}{-2^k\log\l}
\]
increases pointwise, and this completes the proof.
We can write
\[
H(\mu_\l;\l^{2^{k+1}}|1)=H(\mu_\l;\l^{2^{k+1}}|\l^{2^k})
+H(\mu_\l;{\l^{2^{k}}}|1).
\]
Let $\l^{2^k}/2\le r\le \l^{2^k}$ be such that $r/\l^{2^{k+1}}\in\Z_{>0}$.
We apply \eqref{eq:subscale}, then Lemmata \ref{lm:entropy-conv-nondecrease} and \ref{lemma:Lipschitz}
and then \eqref{eq:scaling}
to the first term on the right hand side and obtain
\begin{align*}
H(\mu_\l;\l^{2^{k+1}}|\l^{2^k})
\ge& H(\mu_\l;\l^{2^{k+1}}|r)
\ge H(\mu_\l^{(0,\l^{2^k}]};\l^{2^{k+1}}|r)\\
\ge& H(\mu_\l^{(0,\l^{2^k}]};\l^{2^{k+1}}|\l^{2^k})-2
=H(\mu_\l;{\l^{2^{k}}}|1)-2.
\end{align*}
Combining our estimates, we find
\begin{align*}
f_{k+1}(\l)
=&\frac{H(\mu_\l;\l^{2^{k+1}}|1)-2}{-2^{k+1}\log \l}
\ge \frac{2 (H(\mu_\l;\l^{2^{k}}|1)-2)}{-2^{k+1}\log \l}
=f_k(\l),
\end{align*}
as required.
\end{proof}

\section{Initial bounds on entropy using Diophantine considerations}\label{sc:initial}

The purpose of this section is to prove the following two results, which provide
the initial lower bounds on the entropy of Bernoulli convolutions that we will bootstrap in the next section.

\begin{theorem}\label{th:alg-approx}
For every $\e>0$, there is a number $c>0$ such that the following holds for all $n$ large enough (depending only on $\e$).
Let $0<r<n^{-3n}$ and $0<\l\leq 1-\e$ be numbers.
Suppose $H(\mu_\l^{(\l^n,1]};r)<n$.

Then there is an algebraic number $\eta$ that is a root of a polynomial in $\cP_n$ such that
\[
|\eta-\l|< r^c
\]
and
\[
h_\eta\le \frac{H(\mu_\l^{(\l^n,1]};r)}{n}.
\]
\end{theorem}

Recall that
\[
h_\eta=\lim_{n\to\infty}\frac{H(\mu_\eta^{(\eta^n,1]})}{n}=\inf \frac{H(\mu_\eta^{(\eta^n,1]})}{n}.
\]

\begin{theorem}\label{th:alg-approx2}
For every $\e>0$, there is a number $c>0$ such that the following holds for all $n$ large enough (depending only on $\e$).
Let  $0<\l\leq 1-\e$ be a number.
Suppose that there is an algebraic number $\eta$ that is a root of a polynomial in $\cP_n$ and $|\l-\eta|<n^{-4n}$.

Then
\[
H(\mu_\l^{(\l^n,1]};r)= n
\]
for all $r\le|\l-\eta|^{1/c}$.
\end{theorem}

\begin{remark}\label{rm:c-improve}
We note that the constant $c$ in both theorems can be taken independent of $\e$, and in fact, arbitrarily close to
$1$, provided we assume $0<r<n^{-Bn}$ in Theorem \ref{th:alg-approx} and $|\l-\eta|<n^{-Bn}$ in
Theorem \ref{th:alg-approx2} for some suitably large $B$ depending on $\e$.
\end{remark}

We outline the main idea behind the proofs of these theorems.
If $H(\mu_\l^{(\l^n,1]};r)$ is ``small'', then there are ``many'' choices of signs $a_i,b_i\in\{-1,1\}$ such that
\[
\frac{1}{2}|(a_0\l^0+\ldots+a_{n-1}\l^{n-1})-(b_0\l^0+\ldots+b_{n-1}\l^{n-1})|<r/2.
\]
Observe that the expression on the left hand side is (the absolute value of) a polynomial in $\l$ of degree at most
$n-1$ with coefficients in $\{-1,0,1\}$.

In the next proposition, we consider a collection of such polynomials that take ``small'' values at $\l$
and conclude that they have a common zero $\eta$ near $\l$.
To prove Theorem \ref{th:alg-approx}, we will use this to estimate the Shannon entropy of $\mu_\eta^{(\eta^n,1]}$
and conclude that $h_\eta$ is small.

\begin{proposition}\label{pr:alg-approx2}
For every $\e>0$, there is a number $c>0$ such that the following holds for all $n$ large enough (depending only on $\e$).
Let $A\subset \cP_n$ be a set of polynomials and let $0<r<n^{-3n}$ and $\l\in\C$ be numbers.
Suppose $\e<|\l|<1-\e$ and $|P(\l)|\le r$ for all $P\in A$.

Then there is a number $\eta\in \C$ such that
$P(\eta)=0$ for all $P\in A$ and
\[
|\eta-\l|\le r^c.
\]
\end{proposition}

This proposition will be proved using a B\'ezout identity expressing the greatest common divisor $D$
of the elements of $A$ as
\[
D=Q_1P_1+\ldots+Q_mP_m,
\]
where $P_i\in A$ and $Q_i\in\Q[x]$ whose degree and coefficients are controlled.
We will then argue that $D$ must be ``small'' at $\l$, hence it must have a zero near $\l$.

To deduce Theorem \ref{th:alg-approx2}, we will exploit the fact that the roots of the polynomials in $\cP_n$
repel each other.
If $\l$ can be approximated by a root $\eta$ of  a polynomial in $\cP_n$ with ``very small'' error, then this approximation
is unique.
If we set the scale $r$ smaller than $|\l-\eta|^{1/c}$ with the constant $c$ from Theorem \ref{th:alg-approx},
then that theorem implies the claim.

The result that we use about the separation between roots of polynomials in $\cP_n$ is the following one due to Mahler.

\begin{theorem}[Mahler]\label{th:Mahler}
Let $n\ge 9$.
Let $\eta\neq\eta'$ be two algebraic numbers each of which is a root of a polynomial in $\cP_n$.
Then $|\eta-\eta'|>2n^{-4n}$.
\end{theorem}

\begin{proof}
Let $P\in\Z[X]$ of degree $d$.
By Mahler's result \cite{Mah-discriminant}*{Theorem 2},
it follows that the distance between any two distinct roots of $P$
is at least
\[
\sqrt{3}d^{-(d+2)/2}M(P)^{-(d-1)},
\]
where $M(P)$ is the Mahler measure of $P$.

If $\eta$ and $\eta'$ are Galois conjugates, then we take $P$ to be
their minimal polynomial.
If they are not Galois conjugates, then we take $P$ to be the product of their
minimal polynomials.

In either case, the degree of $P$ is at most $2n$, and its Mahler measure
is at most the product of the Mahler measures of the polynomials in $\cP_n$
whose roots $\eta$ and $\eta'$ are.
By \cite{BG-heights}*{Lemma 1.6.7}, we have $M(P)\le n+1$.
Therefore, we have
\[
|\eta-\eta'|\ge \sqrt{3} (2n)^{-n-1}(n+1)^{-2n+1}>(2n)^{-3n}>2n^{-4n},
\]
provided $n\ge 9$.
\end{proof}

Finally, we note that Theorem \ref{th:Mahler} offers an alternative way to prove
a weaker version of Proposition \ref{pr:alg-approx2}.
Indeed, one can argue that any $P\in A$ must have a zero near $\l$, because $P(\l)$ is ``small''.
Then one may use Theorem \ref{th:Mahler} to conclude that these zeros must coincide.

However, our argument based on the B\'ezout identity has the advantage that it gives a similar result
(with weaker approximation) even without the hypothesis $|\l|<1-\e$.
We formulate this below in Proposition \ref{pr:alg-approx1}.
Although that result is not required for the proof of Theorem \ref{th:main}, we find it of independent interest.

In addition, our approach based on the B\'ezout identity could be used to give an alternative proof of Theorem \ref{th:Mahler}
with a worse constant, but we do not pursue this here.

The rest of this section is organized as follows.
We formulate and prove the B\'ezout identity in Section \ref{sc:Bezout}.
Section \ref{sc:gcd} is devoted to the proof of Proposition \ref{pr:alg-approx2}.
Finally, we prove Theorems \ref{th:alg-approx} and \ref{th:alg-approx2} in Section \ref{sc:aaproof}.

\subsection{}\label{sc:Bezout}

The purpose of this section is to prove the following result.

\begin{proposition}\label{pr:nullstellensatz}
Let $A\subset \cP_n$ be a set of polynomials and let $D$ be their greatest common divisor in $\Z[x]$.
Then there is a number $m\le n+1$ and polynomials $P_1,\ldots, P_m\in A$ and $Q_1,\ldots, Q_m\in \Q[x]$
such that
\[
D=\sum_{j=1}^m Q_jP_j
\]
and
\[
\deg(Q_j)\le n-1, \quad h(Q_j)\le 2^n(2n)!
\]
for all $j$.
\end{proposition}

Here and everywhere below, $h(Q)$ denotes the naive height, the
maximum of the numerators
and denominators of the coefficients of $Q$.
We begin with some preliminary observations.

\begin{lemma}\label{lm:height}
Let $D\in\Z[x]$ be a polynomial that divides a polynomial $P\in \cP_n$
for some $n\in \Z_{\ge 1}$.
Then $l^{1}(D)\le 2^n n$.
\end{lemma}

Here, and everywhere below, $l^{p}(D)$ denotes the $l^p$ norm of the vector
formed from the coefficients of $D$.

\begin{proof}
By  \cite{BG-heights}*{Lemma 1.6.7}, we have
\[
M(D)\le M(P)\le (\deg(P)+1)^{1/2}h(P)\le (n+1)^{1/2}\le n.
\]

We also have (see \cite{Mah-Jensen}*{Equation (4)})
\[
l^1(D)\le 2^n M(D)\le 2^n n,
\]
which was to be proved.
\end{proof}

\begin{lemma}\label{lm:lineq}
Let $n\in \Z_{>0}$ and let $v_1,\ldots, v_N\in \{-1,0,1\}^n$ be vectors.
Suppose that $w\in \Z^n$ is in the $\Q$-span of $v_1,\ldots, v_N$.
Then there are rational numbers $\l_1,\ldots, \l_N$ such that
\begin{equation}\label{eq:lineq}
w=\l_1 v_1+\ldots+\l_N v_N,
\end{equation}
at most $n$ of the $\l_i$ are non-zero and their numerators
and denominators are bounded in absolute value by $\max(n!,l^1(w)(n-1)!)$.
\end{lemma}

\begin{proof}
We select a non-zero minor of maximal rank from the matrix $[v_1,\ldots, v_N]$ and then
solve the equation using Cramer's rule.

The rank is at most $n$, hence the number of non-zero $\l_i$ is indeed at most $n$.
The non-zero $\l_i$ are the ratio of two determinants of rank at most $n$.
In the denominator all entries come from the entries of $v_i$, hence they are $-1$, $0$ or $1$.
This determinant is clearly bounded by $n!$.

The entries of the numerator are similarly $-1$, $0$ or $1$ except for one
column whose entries come from $w$.
Expanding the determinant in that column we obtain the bound $l^1(w)(n-1)!$.
\end{proof}

\begin{proof}[Proof of Proposition \ref{pr:nullstellensatz}]
By the Nullstellensatz or simply by the Euclidean algorithm, there are polynomials
$P_1,\ldots, P_m\in A$ and $Q_1,\ldots, Q_m\in \Q[x]$ such that
\begin{equation}\label{eq:NSS}
D=\sum_{j=1}^m Q_jP_j.
\end{equation}

We may assume that the polynomials $P_j$ are linearly independent.
Indeed, we could achieve this situation by expressing some of the polynomials $P_j$
that appear in \eqref{eq:NSS} by linear combinations of others.
This yields $m\le n+1$.

We may also assume that $\deg(Q_j)< \deg(P_m)\le n$ for all $j<m$.
Indeed, if this was false for some $j<m$, we can write
$Q_j=Q_j'P_m+R_j$ and replace $Q_j$ by $R_j$ and $Q_m$ by $Q_m+Q_j'P_j$.
This substitution does not change the value of \eqref{eq:NSS}, since
\[
R_jP_j+(Q_m+Q_j'P_j)P_m=(Q_j-Q_j'P_m)P_j+(Q_m+Q_j'P_j)P_m=Q_jP_j+Q_mP_m.
\]
These substitutions can be executed simultaneously without affecting each other.

We observe that
\begin{align*}
\deg(Q_mP_m)\le &\max(\deg(D),\deg(Q_1P_1),\ldots, \deg(Q_{m-1}P_{m-1}))\\
< &\deg(P_m)+n,
\end{align*}
which in turn gives $\deg(Q_m)<n$.

We write
\[
P_j=p_{j,n}x^n+\ldots+p_{j,0},\quad D=d_nx^n+\ldots+d_0,
\]
where we allow $p_{j,n}=0$ and $d_n=0$.
We consider the vectors
\[
v_{j,k}=(\underbrace{0,\ldots,0}_{k},p_{j,0},\ldots,p_{j,n},\underbrace{0,\ldots,0}_{n-1-k})\in\{-1,0,1\}^{2n}
\]
for $j=1,\ldots,m$ and $k=0,\ldots, n-1$ and
\[
w=(d_0,\ldots,d_n,0,\ldots, 0)\in \Z^{2n}.
\]

By \eqref{eq:NSS}, $w$ is in the $\Q$-span of the vectors $v_{j,k}$.
By Lemma \ref{lm:height}, we have $l^1(w)\le 2^nn$.
We apply Lemma \ref{lm:lineq} to find rational numbers $\l_{j,k}$
with numerators and denominators bounded by $n2^n(2n-1)!$ such that
\[
w=\sum\l_{j,k}v_{j,k}.
\]

We conclude the proof by replacing $Q_{j}$ by $\sum\l_{j,k}x^k$.
\end{proof}

\subsection{}\label{sc:gcd}

The purpose of this section is to prove Proposition \ref{pr:alg-approx2} and its following
variant.

\begin{proposition}\label{pr:alg-approx1}
Let $n\in\Z$ be sufficiently large (larger than an absolute constant), let $A\subset \cP_n$ be a set of polynomials and let $0<r<(2n)^{-2n}$ and $\l\in\C$ be numbers.
Suppose $|P(\l)|\le r$ for all $P\in A$.

Then there is a number $\eta\in \C$ such that
$P(\eta)=0$ for all $P\in A$ and
\[
|\eta-\l|< r^{1/n}(2n)^2.
\]
\end{proposition}

We give a bound on the number of roots a polynomial in $\cP_n$ may have
away from the unit circle using Jensen's formula.
This will be used in the proof of Proposition \ref{pr:alg-approx2}
to show that such a polynomial can take very small values only near its roots.

\begin{lemma}\label{lm:Jensen}
There is a function $a(k):\Z_{>0}\to (0,1)$ such that $\lim_{k\to\infty}a(k)=1$
and the following holds.
Let $P\in \cP_n$ be a non-zero polynomial for some $n\in \Z_{\ge 0}$.
Then there are at most $k$ non-zero roots of $P$ of absolute value less than $a(k)$.
\end{lemma}

This result is not new, see e.g. \cite{BBBP}.

\begin{proof}
Without loss of generality we may assume that $|P(0)|=1$.
Indeed, we may divide $P$ by an appropriate power of $x$ to obtain a new polynomial that has this property.
We prove the lemma taking
\[
a(k)=\frac{k}{k+1}\cdot\frac{1}{(k+1)^{1/k}}.
\]

We denote by $z_1,\ldots, z_K$ the roots of $P$ of absolute value less than
$a(k)$.
We set $r=k/(k+1)$ and apply Jensen's formula on the disk of radius $r$:
\[
\sum_{j=1}^K\log\frac{|r|}{|z_j|}\le \int_0^1\log|P(r e^{2\pi i t})| dt
\]

We note that
\[
|P(z)|\le 1+|z|+|z|^2+\ldots=\frac{1}{1-r}=k+1
\]
for all $z$ with $|z|=r$.

Thus
\[
K\cdot\log\big((k+1)^{1/k}\big)\le \log(k+1),
\]
which yields $K\le k$, as claimed.
\end{proof}

\begin{proof}[Proof of Propositions \ref{pr:alg-approx1} and \ref{pr:alg-approx2}]
We begin with  Proposition \ref{pr:alg-approx1}.
We denote by  $D$ the greatest common divisor of the polynomials in $A$. Note that the hypothesis (when $A$ is non-empty) implies that $|\lambda| \leq 2$. 
We use Proposition \ref{pr:nullstellensatz} and the fact that $|P(\l)|\le r$ for all $P\in A$.
We get
\begin{equation}\label{D-bound}
|D(\l)|\le (n+1) 2^{2n+1}(2n)!\cdot r< (2n)^{2n}\cdot r.
\end{equation}

Since $|D(\l)|<1$ and $D$ has integer coefficients, $D$ is not constant.
We denote by $\eta_1,\ldots, \eta_d$ the  roots of $D$ taking multiplicities into account.
Then
\[
|D(\l)|=\prod_{j=1}^d|\eta_j-\l|,
\]
hence there is some $j$ such that
\[
|\eta_j-\l|\le |D(\l)|^{1/d}< r^{1/n}(2n)^2,
\]
as claimed.

To prove  Proposition \ref{pr:alg-approx2}, we apply Lemma \ref{lm:Jensen} and find
that there is a number $k$ depending only on $\e$ such that any polynomial in $\cP_n$
has at most $k$ non-zero roots of modulus at most $1-\e/2$.
Since $D$ divides such a polynomial, the same bound holds for its roots.

We denote by $\eta_1,\ldots, \eta_{l}$ the non-zero roots of $D$ of modulus at most $1-\e/2$.
Then $l\le k$ and
\[
|D(\l)|\ge (\e/2)^{\deg D-l}\prod_{j=1}^l|\eta_j-\l|.
\]
Thus there is some $j$ such that
\[
|\eta_j-\l|^l\le |D(\l)|\cdot(\e/2)^{-n}.
\]

Since $r<n^{-3n}$, we have from $(\ref{D-bound})$,
\[
|D(\l)|\cdot(\e/2)^{-n}\le (4n/\e)^{2n}\cdot r< r^{1/10},
\]
if $n$ is large enough.
Hence
\[
|\eta_j-\l|< r^{1/(10l)},
\]
as required.
\end{proof}

\begin{remark}
The constant $c$ in Proposition \ref{pr:alg-approx2} can be taken arbitrarily close to $1$
if $r<n^{-Bn}$ for $B$ suitably large.

Indeed, in the setting of the above proof, denote by $\eta_j$ a root of $D$ of minimal distance to $\l$
among $\eta_1,\ldots, \eta_l$.
By Theorem \ref{th:Mahler}, there is at most one root at distance at most $n^{-4n}$ from $\l$, hence
$|\l-\eta_i|\ge n^{-4n}$ for all $i\neq j$.
From this, we obtain
\[
|\l-\eta_j|\le n^{4kn}(\e/2)^{-n} D(\l)\le r^c,
\]
where $c$ can indeed be taken arbitrarily close to $1$, provided $r$ is as small as we assumed above.

We will see in the next section that the constant $c$ in Theorems \ref{th:alg-approx} and \ref{th:alg-approx2}
are the same as in Proposition \ref{pr:alg-approx2}, hence this justifies the claims made in Remark \ref{rm:c-improve}.
\end{remark}

\subsection{}\label{sc:aaproof}

\begin{proof}[Proof of Theorem \ref{th:alg-approx}]
Let $\xi_0,\ldots, \xi_{n-1}$ be a sequence of independent unbiased $\pm1$-valued random variables.
Let $t\in \R$ be such that
\[
H(\mu_\l^{(\l^n,1]};r)\ge H\Big(\Big\lfloor r^{-1}\sum_{j=0}^{n-1}\xi_j\l^j+t\Big\rfloor\Big).
\]
For each $a\in \Z$ let
\[
\Omega_a=\Big\{(\omega_0,\ldots,\omega_{n-1})\in\{-1,1\}^n:
\Big\lfloor r^{-1}\sum_{j=0}^{n-1}\omega_j\l^j+t\Big\rfloor=a\Big\}.
\]
We note the identity
\[
H\Big(\Big\lfloor r^{-1}\sum_{j=0}^{n-1}\xi_j\l^j+t\Big\rfloor\Big)
=\sum_{a\in\Z}\frac{|\Omega_a|}{2^n}\log\Big(\frac{2^n}{|\Omega_a|}\Big).
\]
In particular $|\Omega_a|>1$ for at least one $a\in\Z$, because $H(\mu_\l^{(\l^n,1]})<n$.

We consider the set of polynomials
\[
A=\bigcup_{a\in \Z}\Big\{\sum_{j=0}^{n-1}\frac{\omega_j-\omega_j'}{2}x^j:\omega\neq\omega'\in\Omega_a \Big\}.
\]
Since $|\Omega_a|>1$ for at least one $a\in\Z$, $A$ is not empty.
We observe that $P\in \cP_n$ and $|P(\l)|\le r$ for each $P\in A$.
We apply Proposition \ref{pr:alg-approx2} and find $\eta\in \C$ such that $|\eta-\l|\le r^c$
and $P(\eta)=0$ for all $P\in A$.

For any $a\in \Z$ and $\omega,\omega'\in \Omega_a$, we have
\[
\sum_{j=0}^{n-1}\frac{\omega_j-\omega_j'}{2}\eta^j=0,
\]
hence
\[
\sum_{j=0}^{n-1}\omega_j\eta^j=\sum_{j=0}^{n-1}\omega_j'\eta^j.
\]
Thus
\[
H\Big(\sum_{j=0}^{n-1}\xi_j\eta^j\Big)
\le \sum_{a\in\Z}\frac{|\Omega_a|}{2^n}\log\Big(\frac{2^n}{|\Omega_a|}\Big).
\]

We combine our inequalities to obtain
\[
H(\mu_\l^{(\l^n,1]};r)\ge H\Big(\sum_{j=0}^{n-1}\xi_j\eta^j\Big)\ge n h_\eta.
\]
Recall $h_n=\inf H(\mu_\l^{(\l^n,1]})/n$.
\end{proof}

\begin{proof}[Proof of Theorem \ref{th:alg-approx2}]
The constant $c$ is the same as in Theorem \ref{th:alg-approx}.
We suppose to the contrary that $H(\mu_\l^{(\l^n,1]};r)<  n$
and apply that theorem.

We find an algebraic number $\eta'$, which is a root of a polynomial in $\cP_n$
such that $|\l-\eta'|<r^c\le|\l-\eta|$.
In particular, $\eta'\neq\eta$.
Moreover,
\[
|\eta-\eta'|\le|\l-\eta'|+|\l-\eta|<2n^{-4n},
\]
which contradicts Theorem \ref{th:Mahler}.
\end{proof}

\section{Increasing entropy of convolutions}\label{sc:convolution}

In this section, we apply Theorem \ref{th:low-entropy-convolution}
to improve on the entropy estimates that we obtained in the previous section.
We begin with two preliminary results in the next two sections and conclude the proof of Theorem \ref{th:main}
in Section \ref{sc:proof-th-main}

\subsection{}\label{sc:first-stage}

The purpose of this section is the following proposition.

\begin{proposition}\label{pr:first-stage}
Let $1/2< \l<1$ and  $n,K,\a>0$ be numbers, with $K\geq 10$.
Suppose
\begin{align}
H(\mu_\l; r|2r)\le& 1-\a \quad\text{for all $r>0$},\label{eq:first-stage1}\\
H(\mu_\l^{(\l^n,1]};\l^{Kn}|\l^{10n})\ge& \a n.\label{eq:first-stage2}
\end{align}
Suppose further $n>C_0(\log K)^2$, where $C_0$ is a suitably large number depending only on $\a$ and $\l$.

Then, there are numbers $R_1,\ldots, R_k$ and $a_1,\ldots, a_k$
such that
\begin{align*}
\l^{-9n}\le R_i\le \l^{-Kn},\quad R_{i+1}&\ge R_i^2,\quad  a_i\ge \frac{c}{\log K},\quad\sum_{i=1}^k a_i\ge c,\\
H(\mu_\l^{(R_ir,R_i^2r]};r|Ar)&\ge \frac{ca_i}{\log K}\log A
\end{align*}
for each $i=1,\ldots, k$, for any $r\le\l^{2Kn}$ and for any $\max(\l^{-2},2)\le A\le\l^{-n}$,
where $c$ is a constant that depends only on $\a$ and $\l$.
\end{proposition}

In the proof of Theorem \ref{th:main}, we fix a parameter $\l$ such that $\dim \mu_\l<1$.
By Lemma \ref{lm:dimlessthan1}, this implies that \eqref{eq:first-stage1} holds at all scales.
Furthermore, we will show that  \eqref{eq:first-stage2} also holds for the appropriate choice of
$n$ and $K$.
To this end, we will use the results of Section \ref{sc:initial}.
In Section \ref{sc:second-stage}, we refine the conclusion of this proposition by further applications of Theorem
\ref{th:low-entropy-convolution}.

We begin the proof of the proposition with a technical lemma.
If we have a bound for the entropy of $\mu^{(a,b]}$ between some scales,
then we can use the scaling identity \eqref{eq:scaling} to obtain bounds for $\mu^{(ar,br]}$
between some other scales.
We take this idea a step further in the next lemma, which will be used
in the proof of Proposition \ref{pr:first-stage} to construct measures,
to which we can apply Theorem \ref{th:low-entropy-convolution}.

\begin{lemma}\label{lm:rescale}
Let $a_1,a_2,b_1,b_2,r_1,r_2,s_1,s_2$ be numbers such that the following holds
\begin{align*}
0\le& a_i<b_i\le 1,\quad 0\le r_i<s_i \quad \text{for $i=1,2$}\\
\l^{-1}s_1/s_2\le& a_1/a_2\\
\l r_1/r_2\ge&b_1/b_2\\
\max(2,\l^{-2})\le&s_2/r_2\le s_1/r_1.
\end{align*}
Suppose
\[
H(\mu_\l^{(a_1,b_1]};r_1|s_1)\ge \b\log(s_1/r_1)
\]
for some $\b\ge 0$.
Then
\[
H(\mu_\l^{(a_2,b_2]};r_2|s_2)\ge \frac{\b}{6}\log(s_2/r_2).
\]
\end{lemma}

We comment on the inequalities imposed in the lemma, which
may look unmotivated on first reading.
They are designed to ensure that for any scaling factor $t$,
the inclusion of scales $t[r_2,s_2]\subset [\l r_1,\l^{-1}s_1]$
implies $(a_1,b_1]\subset t(a_2,b_2]$.

\begin{proof}
We choose a sequence of integers $k_1>\ldots>k_N$ such that
the intervals $[\l^{k_i}r_2,\l^{k_i}s_2]$ cover $[r_1,s_1]$, i.e. we have
\[
\l r_1\le \l^{k_1}r_2\le r_1,\quad s_1\le\l^{k_N}s_2\le\l^{-1}s_1,
\]
and
\[
\l^{k_i}s_2\ge \l^{k_{i+1}} r_2
\]
holds for all $1\le i\le N-1$.

We may choose the sequence in such a way that the overlaps
between the intervals $[\l^{ki}r_2,\l^{ki}s_2]$ are minimal, so that
\[
\l^{k_{i+1}}r_2\ge \l^{k_i+1}s_2
\]
for all $i\le N-2$.
If this is the case, we have
\[
\l^{k_{i+1}}s_2\ge \l^{k_i}s_2\cdot(\l s_2/r_2)
\]
and then
\[
\l^{k_i}s_2 \ge \l^i r_1\cdot(s_2/r_2)^i\ge r_1\cdot (s_2/r_2)^{i/2}
\]
follows for $1\le i\le N-1$ by induction.
Clearly, we may assume $\l^{k_{N-1}}s_2<s_1$, since otherwise we would not need the interval $[\l^{k_N}r_2,\l^{k_N}s_2]$
to cover $[r_2,s_2]$.
Hence we may assume that
\[
N\le 2\frac{\log(s_1/r_1)}{\log(s_2/r_2)}+1\le 3\frac{\log(s_1/r_1)}{\log(s_2/r_2)}.
\]

Using \eqref{eq:subscale}, we write
\begin{align*}
\sum_{i=1}^{N} H(&\mu_\l^{(a_1,b_1]};\l^{k_i}r_2|\l^{k_i}s_2)\\
\ge&H(\mu_\l^{(a_1,b_1]};r_1|\l^{k_2}r_2)
+\sum_{i=2}^{N-1} H(\mu_\l^{(a_1,b_1]};\l^{k_i}r_2|\l^{k_{i+1}}r_2)\\
&+H(\mu_\l^{(a_1,b_1]};\l^{k_N}r_2|s_1)\\
=&H(\mu_\l^{(a_1,b_1]};r_1|s_1)\ge \b\log(s_1/r_1).
\end{align*}
Thus there is some $i$ such that
\[
H(\mu_\l^{(a_1,b_1]};\l^{k_i}r_2|\l^{k_i}s_2)\ge\frac{\b}{N}\log(s_1/r_1)\ge
\frac{\b}{3}\log(s_2/r_2).
\]

Using $\l^{k_i}r_2\ge\l r_1$, $\l^{k_i}s_2\le \l^{-1}s_1$ and the assumptions
in the statement of the lemma, we have
\[
b_1/b_2\le \l r_1/r_2 \le \l^{k_i}\le \l^{-1}s_1/s_2\le a_1/a_2,
\]
hence
\[
(\l^{k_i}a_2,\l^{k_i}b_2]\supset (a_1,b_1].
\]
Therefore, we can use \eqref{eq:scaling} and \eqref{eq:subset-noninteger} and write
\begin{align*}
H(\mu_\l^{(a_2,b_2]};r_2|s_2)
=&H(\mu_\l^{(\l^{k_i}a_2,\l^{k_i}b_2]};\l^{k_i}r_2|\l^{k_i}s_2)\\
\ge& \frac{1}{2}H(\mu_\l^{(a_1,b_1]};\l^{k_i}r_2|\l^{k_i}s_2)
\ge\frac{\b}{6}\log(s_2/r_2).
\end{align*}
\end{proof}

\begin{proof}[Proof of Proposition \ref{pr:first-stage}]
Write
\[
J=\Big\lceil \frac{\log(K/10)}{\log(11/10)}\Big\rceil.
\]
Then $\l^{Kn}>\l^{11n(11/10)^{J-1}}$, hence
\[
H(\mu_\l^{(\l^n,1]};\l^{11(11/10)^{J-1}n}|\l^{10n})\ge \a n.
\]
For each integer $0\le j<J$ define $b_j$ by
\begin{equation}\label{eq:Jbound}
H(\mu_\l^{(\l^n,1]};\l^{11(11/10)^jn}|\l^{10(11/10)^jn})=b_j n.
\end{equation}
Then $b_0+\ldots + b_{J-1}\ge \a$.

We fix a $j\in\{0,\ldots, J-1\}$.
Put $T=\lfloor (11/10)^j\rfloor$.
We note the identity
\[
\mu_\l^{(\l^{Tn},1]}=\mu_\l^{(\l^{Tn},\l^{(T-1)n}]}*\mu_\l^{(\l^{(T-1)n},\l^{(T-2)n}]}*\ldots*\mu_\l^{(\l^{n},1]},
\]
and set out to apply Theorem \ref{th:low-entropy-convolution} and find a lower
bound on the entropy of $\mu_\l^{(\l^{Tn},1]}$ between suitably chosen scales.

Scaling \eqref{eq:Jbound} (see \eqref{eq:scaling}), we can write
\[
H(\mu_\l^{(\l^{(t+1)n},\l^{tn}]};\l^{(11(11/10)^j+t)n}|\l^{(10(11/10)^j+t)n})= b_j n.
\]
For $t=0,\ldots,T-1$, we have
\[
[\l^{(11(11/10)^j+t)n},\l^{(10(11/10)^j+t)n}]
\subseteq [\l^{12(11/10)^jn},\l^{10(11/10)^jn}],
\]
hence we can use \eqref{eq:subscale} to get
\[
H(\mu_\l^{(\l^{(t+1)n},\l^{tn}]};\l^{12(11/10)^j n}|\l^{10(11/10)^jn})\ge b_j n.
\]

We can now apply Theorem \ref{th:low-entropy-convolution}  repeatedly $T$ times
for $t=0,\ldots, T-1$ with
\begin{align*}
\mu=&\mu_\l^{(\l^{tn},\l^{(t-1)n}]}*\ldots*\mu_\l^{(\l^{n},1]}\quad
\text{or $\mu=\d_0$ if $t=0$},\\
\nu=&\mu_\l^{(\l^{(t+1)n},\l^{tn}]},\\
\b=&\frac{b_j}{2(11/10)^j\log \l^{-1}}\ge \frac{c b_j}{T},
\end{align*}
where $c>0$ is a constant depending only on $\l$.
We obtain
\[
H(\mu_\l^{(\l^{Tn},1]};\l^{12(11/10)^j n}|\l^{10(11/10)^jn})
\ge T\Big(\frac{cb_j}{T\log(T/b_j)} \log(\l^{-2(11/10)^jn})-C\Big).
\]

Assume that $j$ is such that $b_j\ge \a/2J$.
Since $j\le J-1$, the definitions  of $T$ and $J$ yield
$T\le K$ and $b_j\ge c/\log K$ for some other constant depending only on $\a$ and $\l$, which we keep denoting by $c$ by abuse of notation.
Since we assumed $n>C_0(\log K)^2$ for any fixed number $C_0$ depending on $\a$ and $\l$,
the term $TC$ becomes negligible.
Thus
\begin{equation}\label{eq:first-stage-rough}
H(\mu_\l^{(\l^{Tn},1]};\l^{12(11/10)^j n}|\l^{10(11/10)^jn})
\ge c\frac{ b_j}{\log K}  \log(\l^{-2(11/10)^jn}).
\end{equation}

Now we combine \eqref{eq:first-stage-rough} with Lemma \ref{lm:rescale}.
To that end, we need to choose a number $Q_j$ in such a way that the
following inequalities are satisfied:
\begin{align}
\l^{-1}\frac{\l^{10(11/10)^jn}}{Ar}&\le\frac{\l^{Tn}}{Q_jr},\label{eq:fs-cond1}\\
\l\frac{\l^{12(11/10)^jn}}{r}&\ge\frac{1}{Q_j^2r},\label{eq:fs-cond2}\\
A&\le\l^{-2(11/10)^jn}.\label{eq:fs-cond3}
\end{align}

Since \eqref{eq:fs-cond3} always holds when $A\le \l^{-n}$ (which we assumed in the statement of the proposition), we need to consider only the first two conditions.
We observe that the first condition is the most restrictive when $A$ is as small
as possible, hence we may assume $A=\l^{-2}$.
Recall $T\le (11/10)^j$.
Hence \eqref{eq:fs-cond1} and \eqref{eq:fs-cond2} hold if we choose $Q_j$
to satisfy
\begin{align*}
Q_j\le& \l^{-9(11/10)^jn-1}\\
Q_j^2\ge&\l^{-12(11/10)^jn-1}.
\end{align*}
So we can put $Q_j=\l^{-9(11/10)^jn}$ and satisfy these inequalities.

We can now apply Lemma \ref{lm:rescale} to \eqref{eq:first-stage-rough}
with
\begin{align*}
r_1=&\l^{12(11/10)^jn},\;s_1=\l^{10(11/10)^jn},\;a_1=\l^{Tn},\;b_1=1,\\
r_2=&r,\;s_2=Ar,\;a_2=Q_jr,\;b_2=Q_j^2r.
\end{align*}
and write
\[
H(\mu_\l^{(Q_jr,Q_j^2r]}; r|Ar)\ge c \frac{b_j}{6\log K} \log A.
\]
We note $1\le(11/10)^j\le K/10$.
Hence
\[
\l^{-9n}\le Q_j\le\l^{-9Kn/10}<\l^{-Kn}.
\]

Finally we define $a_i=b_{j_i}$ and $R_i=Q_{j_i}$ for a suitably chosen sequence $j_i$.
We first select those $j$ such that $b_j>\a/(2J)$.
This ensures that the above argument applies to all selected $j$ and that $b_j\ge c/\log K$.
We still have $\sum' b_j\ge \a/2$, where $\sum'$ indicates summation over those $j$ that we selected.
Second, we select $j$'s from an arithmetic progression with common difference $10$ such that
 the sum of the selected $b_j$'s are maximal among the possible choices.
Then we still have $\sum'' b_j\ge \a/20$, where $\sum''$ indicates summation over those $j$ that we selected
during the second cut.
Moreover, this choice ensures that $Q_{j'}\ge Q_j^{(11/10)^{10}}>Q_j^2$ if $j'>j$ are two selected indices.
Therefore this subsequence satisfies all the requirements of the proposition.
\end{proof}

\subsection{}\label{sc:second-stage}

In the proof of Theorem \ref{th:main}, we will choose sequences of
suitable parameters $\{n_j\}$ and $\{K_j\}$ such that the conditions of
Proposition \ref{pr:first-stage} hold.
In this section, we consider such sequences and apply Theorem \ref{th:low-entropy-convolution} again together
with the conclusion of Proposition \ref{pr:first-stage}
to obtain even stronger entropy bounds.
Since the entropy between the scales $r$ and $Ar$ cannot be larger than $\log A$, this will
lead to a constraint showing that the sequence $K_j$ has to grow very fast.
In the proof of Theorem \ref{th:main}, this will lead to a contradiction with the hypothesis of that theorem.

\begin{proposition}\label{pr:second-stage}
Let $1/2<\l<1$, $\a>0$ be numbers, let $\{n_j\}_{j=1}^N$ be a sequence of positive integers, and let $\{K_j\}_{j=1}^N$ be a sequence of real numbers each $\geq 10$.

Suppose
\begin{align}
n_{j+1}\ge& K_j n_j\quad\text{for all $j=1,\ldots, N-1$},\label{eq:second-stage0}\\
H(\mu_\l; r|2r)\le& 1-\a \quad\text{for all $r>0$},\label{eq:second-stage1}\\
H(\mu_\l^{(\l^{n_j},1]};\l^{K_jn_j}|\l^{10n_j})\ge& \a n_j\quad\text{for all $j=1,\ldots, N$},\label{eq:second-stage2}\\
n_j\ge&C_0(\log K_j)^2,\label{eq:second-stage3}
\end{align}
where $C_0$ is a sufficiently large number depending only on $\l$ and $\a$.
Suppose further that $n_1$ is sufficiently large so that $\l^{-n_1}\ge\max(2,\l^{-2})$.

Then
\begin{equation}\label{eq:second-stage}
\sum_{j=1}^N\frac{1}{\log K_j\log\log K_j}<C\Big(1+\frac{\sum_{j=1}^N\log K_j}{n_1}\Big),
\end{equation}
where $C$ is a constant that depends only on $\a$ and $\l$.
\end{proposition}

\begin{proof}
Set $A=\l^{-n_1}$ and $r=\l^{2K_Nn_N}$, so $A\ge\max(2,\l^{-2})$.
We apply Proposition \ref{pr:first-stage} with $n=n_j$ and $K=K_j$.
We find numbers $R_{j,i}\in [\l^{-9 n_j},\l^{-K_j n_j}]$ and $a_{j,i}\ge c (\log K_j)^{-1}$ such that
$\sum_{i}a_{j,i}\ge c$ for each $j$ and
\[
H(\mu_\l^{(R_{j,i}r,R_{j,i}^2 r]};r|Ar)\ge \frac{c a_{j,i}}{\log K_j}\log A
\]
for each $j$ and $i$.

We observe that
\[
R_{j+1, \bullet}\ge \l^{-9 n_{j+1}}\ge \l^{-9 K_j n_j}> R_{j,\bullet}^2
\]
for $j=1,\ldots, N-1$.
We also recall $R_{j,i+1}\ge R_{j,i}^2$ from  Proposition \ref{pr:first-stage}.
Thus the intervals $(R_{j,i}r,R_{j,i}^2 r]$ are disjoint.

This means that we can write
\[
\mu_\l= \nu * \Asterisk_{j=1}^N\Asterisk_{i}\mu_\l^{(R_{j,i}r,R_{j,i}^2 r]},
\]
for some probability measure $\nu$.
We can then apply Theorem \ref{th:low-entropy-convolution} repeatedly
with
\[
\nu=\mu_\l^{(R_{j,i}r,R_{j,i}^2 r]} \quad \text{and} \quad \b=\frac{ca_{j,i}}{\log K_j}
\]
for each $j$ and $i$.
Note $\log \b^{-1}\le C\log\log K_j$, since $a_{j,i}\ge c/(\log K_j)$, where $C$
is a constant that depends only on $\l$ and $\a$.
We obtain
\[
H(\mu_\l;r|Ar)\ge \sum_{j=1}^N\sum_i\Big(\frac{c a_{j,i}}{\log K_j\log\log K_j}\log A - C\Big),
\]
where $c,C>0$ are some numbers that depend only on $\l$ and $\a$.

Since $\sum_i a_{j,i}\ge c$, for each $j$ and the entropy between scales of ratio $A$ cannot
be larger than $2\log A$ (see Lemma \ref{lemma:Lipschitz}), we get
\[
 \sum_{j=1}^N\frac{c}{\log K_j \log\log K_j}\log A <2\log A +C\sum_{j=1}^N\log K_j.
\]
This proves the claim upon dividing both sides by $c\log A$, since
$\log A= \log (\l^{-1}) n_1$.
\end{proof}

\subsection{Proof of Theorem \ref{th:main}}\label{sc:proof-th-main}

Let $1/2<\l<1$ be a number such that $\dim\mu_\l <1$, and
fix a small number $\e>0$ such that $\dim\mu_\l+4\e<1$.
We fix a large number $A$, whose value will be set at the end of the proof depending only on
$\l$ and $\e$.
We assume to the contrary that
there are arbitrarily large integers $n_0$ such that
\begin{equation}\label{eq:separation}
|\eta-\l|> \exp(-\exp(\log n\log^{(3)} n))
\end{equation}
for all $\eta\in E_{n,\dim\mu_\l+4\e}$ for all $n\in [n_0,\exp^{(5)}(\log^{(5)}(n_0)+A)]$.
We show that this leads to a contradiction provided $A$ is a sufficiently large number depending on $\l$ and $\e$.

The assumption $\dim\mu_\l<1$ implies that there is $\a>0$ such that
\begin{equation}\label{eq:dimension-gap}
H(\mu_\l;r|2r)<1-\a
\end{equation}
for all $r$; see Lemma \ref{lm:dimlessthan1}.
In addition, we have
\begin{equation}\label{eq:dim-entropy}
H(\mu_\l;r)\le (\dim \mu_\l+\e)\log r^{-1}
\end{equation}
by Lemma \ref{lm:dim-entropy} for all sufficiently small $r$ (depending on $\e$ and $\l$).
Moreover, \eqref{eq:dimension-gap} and \eqref{eq:dim-entropy} hold
for the measure $\mu_\l^I$ in place of $\mu_\l$ for any $I\subset(0,1]$.
Indeed,
\begin{align*}
H(\mu_\l;r)=&\lim_{N\to\infty} H(\mu_\l;r|Nr),\\
H(\mu_\l^I;r)=&\lim_{N\to\infty} H(\mu_\l^I;r|Nr),
\end{align*}
so $H(\mu_\l^{I};r)\le H(\mu_\l;r)$ follows from \eqref{eq:subset}.

It follows from the work of Hochman \cite{hochman}*{Theorem 1.3} that
\footnote{We could avoid using Hochman's result here if we replaced the number $10$ by
$1+\e$.
If we do this, then Propositions \ref{pr:first-stage} and \ref{pr:second-stage}
and their proofs need to be adjusted accordingly, which would turn the calculations
even more tedious.}
\begin{equation}\label{eq:hoc}
H(\mu_\l^{(\l^n,1]}; \l^{10n}|\l^n)<\e\log(\l^{-1}) n,
\end{equation}
if $n$ is large enough  (depending on $\e$ and $\l$).

We fix an integer $n_0$ such that \eqref{eq:separation} holds, and which is sufficiently large;
we require, in particular, that \eqref{eq:hoc} holds for all $n\ge n_0$,
\eqref{eq:dim-entropy} holds for all $r<\l^{n_0}$ and $\l^{-n_0}\ge \max(2,\l^{-2})$.
We define a sequence of integers $n_1,n_2,\ldots, n_N$ by a recursive procedure.
Suppose that $n_j$ is already defined for some $j\ge 0$ and we choose the value of $n_{j+1}$
as follows.
We take
\[
m=\Big\lceil \frac{4n_{j}\log n_j}{c_0\log \l^{-1}}\Big\rceil,
\]
where $c_0$ denotes the minimum of the constants $c$ from Theorems \ref{th:alg-approx} and \ref{th:alg-approx2}
applied with $1-\l$ in the role of $\e$.

We consider two cases.
First, suppose
\begin{equation}\label{eq:first-case}
H(\mu_\l^{(\l^m,1]};m^{-4m/c_0})\ge m\log(\l^{-1})(\dim\mu_\l+3\e).
\end{equation}
In this case, we have
\begin{align*}
H(\mu_\l^{(\l^{m},1]};&m^{-4m/c_0}|\l^{10 m})\\
=& H(\mu_\l^{(\l^m,1]};m^{-4m/c_0})-H(\mu_\l^{(\l^m,1]};\l^m)
-H(\mu_\l^{(\l^m,1]};\l^{10 m}|\l^m)\\
\ge& m\log(\l^{-1})(\dim\mu_\l+3\e) -m\log(\l^{-1})(\dim\mu_\l+\e)-\e m\log (\l^{-1})\\
\ge& \e m \log (\l^{-1}).
\end{align*}
We used \eqref{eq:first-case}, \eqref{eq:dim-entropy} and \eqref{eq:hoc}.
To estimate $H(\mu_\l^{(\l^m,1]};\l^m)$, we used \eqref{eq:dim-entropy}.
In this case,  we set $n_{j+1}=m$.

Second,  suppose
\[
H(\mu_\l^{(\l^m,1]};m^{-4m/c_0})< m\log(\l^{-1})(\dim\mu_\l+3\e).
\]
We apply Theorem \ref{th:alg-approx} and find that there is an algebraic number
$\eta$ that is a root of a polynomial in $\cP_m$, $h_\eta<\log(\l^{-1})(\dim\mu_\l+3\e)$
and $|\l-\eta|< m^{-4m}$.
We assume as we may that $n_0$ is sufficiently large that this guarantees
$h_\eta<\log \eta^{-1}(\dim\mu_\l+4\e)$.
We note $|\eta-\overline\eta|<2 m^{-4m}$, hence $\eta$ is real by Theorem \ref{th:Mahler}.
By Hochman's formula \eqref{eq:Hochman} for the dimension of Bernoulli convolutions
for algebraic parameters, we have $\dim\mu_\eta<\dim\mu_\l+4\e$
and hence $\eta\in E_{m,\dim\mu_\l+4\e}$.

In this case, we set $n_{j+1}$ to be the largest integer $n$ such that $|\l-\eta|<n^{-4n}$.
In particular, $n_{j+1}\ge m$.
It follows from Theorem \ref{th:alg-approx2} applied with $n=n_{j+1}$ and $r=(n_{j+1}+1)^{-4(n_{j+1}+1)/c_0}$
that
\[
H(\mu_\l^{(\l^{n_{j+1}},1]};(n_{j+1}+1)^{-4(n_{j+1}+1)/c_0})=n_{j+1}.
\]
A calculation similar to what we did in the previous case yields
\[
H(\mu_\l^{(\l^{n_{j+1}},1]};n_{j+1}^{-4n_{j+1}/c_0}|\l^{10 n_{j+1}})\ge (1-\log \l^{-1}) n_{j+1},
\]
if $n_0$ is sufficiently large.
(Recall $\dim\mu_\l+4\e<1$.)

We set
\[
K_{j+1}=\frac{4\log(n_{j+1})}{c_0\log \l^{-1}}
\]
and note that
\[
H(\mu_\l^{(\l^{n_{j+1}},1]}; \l^{K_{j+1}n_{j+1}}|\l^{10 n_{j+1}})\ge \e \log(\l^{-1})n_{j+1}
\]
holds in both cases, (provided $\e\log\l^{-1}<1-\log\l^{-1}$, which we may assume).

The choice of $K_{j+1}$ and $m$ in the recursive definition ensures that $n_{j+2}\ge K_{j+1}n_{j+1}$.
Moreover, $n_{j+1}\ge C_0(\log K_{j+1})^2$ also holds with an arbitrarily large constant $C_0$,
provided $n_0$ is sufficiently large.
This means that Proposition \ref{pr:second-stage} is applicable to the sequences $\{n_j\}$ and $\{K_j\}$.
We estimate how fast these sequences may grow.
Let $m$ and $\eta$ be as in the definition of $n_{j+1}$ above.
Suppose that
\begin{equation}\label{eq:m-condition}
m\in [n_0,\exp^{(5)}(\log^{(5)}n_0+A)].
\end{equation}
(We will return to this condition at the end of the proof.)
Then
\[
|\l-\eta|>\exp(-\exp(\log m\log^{(3)}m)).
\]
by the indirect assumption \eqref{eq:separation},
and hence
\[
\exp(n_{j+1})<n_{j+1}^{4 n_{j+1}}<|\l-\eta|^{-1}<\exp^{(2)}(\log m \log^{(3)} m),
\]
which together with $m\le Cn_j\log n_j$ (for some $C$ depending only on $\l$) yields
\[
n_{j+1}<\exp(2\log n_j\log^{(3)} n_j),
\]
provided $n_0$ is sufficiently large.

\noindent{\bf Claim.}
For each $j\ge 0$, we have
\[
n_j<\exp^{(2)}((2j+j_0)\log^{(2)}(2j+j_0)),
\]
where $j_0=\log^{(2)}(n_0)$.

\begin{proof}
The claim is trivial for $j=0$, and we prove the $j>0$ case by induction.
We suppose that the claim holds for some $j$ and prove that it also holds for $j+1$.
We first note
\[
\log^{(3)} n_j< 2\log(2j+j_0).
\]
We can write
\begin{align*}
n_{j+1}<&\exp(2\log n_j\log^{(3)} n_j)\\
<&\exp(2\exp((2j+j_0)\log^{(2)}(2j+j_0))\cdot 2\log(2j+j_0))\\
=&\exp^{(2)}((2j+j_0)\log^{(2)}(2j+j_0)+\log^{(2)}(2j+j_0)+2)\\
<&\exp^{(2)}((2(j+1)+j_0)\log^{(2)}(2(j+1)+j_0)),
\end{align*}
where the last line holds, because we assumed that $n_0$ is large enough, so in particular,
we have $\log^{(2)}(2j+j_0)> 2$.
This proves the claim.
\end{proof}

Using the above claim, we note that for some positive $C_\l$ depending on $\l$ only,
\begin{align*}
\frac{1}{\log K_j\log^{(2)}K_j}
=& \frac{1}{\log(C_\l\log n_j)\log^{(2)}(C_\l\log n_j)}\\
\ge& \frac{1}{2(2j+j_0)\log(2j+j_0)\log^{(2)}(2j+j_0)},
\end{align*}
provided $n_0$ is large enough.

We can write
\begin{align*}
\sum_{j=1}^N&\frac{1}{\log K_j\log\log K_j}\\
\ge& \sum_{j=1}^N\frac{1}{2(2j+j_0)\log(2j+j_0)\log^{(2)}(2j+j_0)}\\
\ge& c (\log^{(3)}(N+j_0)-\log^{(3)} j_0),
\end{align*}
where $c$ is an absolute constant.

We write $B=2C/c$, where $c$ is the above constant and $C$ is the constant from Proposition \ref{pr:second-stage} applied with the minimum of $\e \log \l^{-1}$ and $\a$ in the role of $\a$.
We put
\[
N:=\lfloor\exp^{(3)}(\log^{(3)}(j_0)+B)\rfloor.
\]
Then
\begin{equation}\label{eq:logKloglogK}
\sum_{j=1}^N\frac{1}{\log K_j\log\log K_j}
\ge c(\log^{(3)}N-\log^{(3)}j_0)
\ge cB
\ge 2C.
\end{equation}

On the other hand, we can write
\[
\log K_j\le 2(2j+j_0)\log^{(2)}(2j+j_0) \leq 6N \log^{(2)} (3N) <10N^2
\]
for $j\le N$, if $n_0$ and hence $j_0$ is sufficiently large, and this yields
\[
\sum_{j=1}^N \log K_j\le 10 N^3.
\]
We note that
\[
N\le \exp^{(3)}(\log^{(2)}(j_0))=\exp(j_0)=\log(n_0),
\]
if $n_0$ and hence $j_0$ is sufficiently large.
This and $n_1>n_0$ implies
\[
\frac{\sum_{j=1}^N \log K_j}{n_1}<1,
\]
provided $n_0$ is sufficiently large,
and hence we have a contradiction with \eqref{eq:logKloglogK} and
Proposition \ref{pr:second-stage}.

It remains to verify that the condition \eqref{eq:m-condition} holds each time we used it.
Clearly we always had $n_0\le m\le n_N$.
Since $N\ge j_0$, we have
\begin{align*}
2N+j_0\le& 4 N\le \exp(2+\exp^{(2)}(\log^{(3)}(j_0)+B))\\
\le& \exp^{(2)}(2+\exp(\log^{(3)}(j_0)+B))
\le \exp^{(3)}(\log^{(3)}(j_0)+B+2).
\end{align*}
In addition,
\begin{align*}
(2N+j_0)&\log^{(2)}(2N+j_0)\\
\le& \exp^{(3)}(\log^{(3)}(j_0)+B+2)\exp(\log^{(3)}(j_0)+B+2)\\
\le& \exp(\exp^{(2)}(\log^{(3)}(j_0)+B+2)+\log^{(3)}(j_0)+B+2)\\
\le& \exp(2\exp^{(2)}(\log^{(3)}(j_0)+B+2))\\
\le& \exp^{(3)}(\log^{(3)}(j_0)+B+3).
\end{align*}
Then we have
\[
n_N\le \exp^{(2)}((2N+j_0)\log^{(2)}(2N+j_0))
\le  \exp^{(5)}(\log^{(5)}(n_0)+B+3).
\]
This shows that \eqref{eq:m-condition} holds provided $A\ge B+3$.
This completes the proof of the theorem.

\bibliographystyle{abbrv}
\bibliography{bibfile}

\end{document}